\documentclass[twoside]{amsart}
\usepackage{amsmath}
\usepackage{amsfonts}
\usepackage{amssymb}
\usepackage{enumerate}
\usepackage{a4wide,amsmath,amssymb,latexsym,amsthm}
\usepackage[pagewise]{lineno}
\numberwithin{equation}{section}
\usepackage{mathrsfs,mathtools,epic,bm}
\usepackage{hyperref}
\hypersetup{colorlinks=true,linkcolor=blue,filecolor=mangeta,urlcolor= cyan}

\newtheorem{theorem}{Theorem}[section]
\newtheorem{proposition}[theorem]{Proposition}
\newtheorem{lemma}[theorem]{Lemma}
\newtheorem{corollary}[theorem]{Corollary}
\theoremstyle{definition}
\newtheorem{definition}[theorem]{Definition}
\newtheorem{example}[theorem]{Example}
\newtheorem{remark}[theorem]{Remark}

\newtheorem{assumption}[theorem]{Assumption}

\numberwithin{equation}{section}

\def \dis {\displaystyle}

\def \R {\mathbb{R}}

\def \N {\mathbb{N}}
\def \C {\mathcal{C}}

\def \C {{\mathcal C}}

\def \U {{\mathcal U}}

\def \hvarphi \hat{\varphi}
\def \dx{\mathrm{d}x}

\def \dt{\mathrm{d}t}

\def \dq {\mathrm{d}t \, \mathrm{d}x}

\def\Om{\Omega}

\def \R {\mathbb{R}}

\def \N {\mathbb{N}}

\def \C {\mathcal{C}}

\def \U {\mathcal{U}}
\def \C {{\mathcal C}}

\keywords{Degenerate parabolic systems; optimal control; first- and second-order optimality conditions; weak maximum principle.}

\begin{document}
	\title[Second-order optimality conditions for the bilinear optimal control]{Second-order optimality conditions for the bilinear optimal control of a degenerate equation}
\author{Cyrille Kenne }
\address[Cyrille Kenne]{Laboratoire LAMIA, Universit\'e des Antilles,
	Campus Fouillole, 97159 Pointe-{\`a}-Pitre  Guadeloupe (FWI)-	
	Laboratoire L3MA, UFR STE et IUT, Universit\'e
	des Antilles, 97275 Schoelcher, Martinique}
\email{kenne853@gmail.com}
\author{Landry Djomegne }
\address[Landry Djomegne]{University of Dschang, BP 67 Dschang, Cameroon, West region}
\email{landry.djomegne@gmail.com}

\author{Pascal Zongo }
\address[Pascal Zongo]{Laboratoire L3MA, UFR STE et IUT, Universit\'e des Antilles, 97275 Schoelcher, Martinique}
\email{pascal.zongo@gmail.com}

	\date{\today}

	\begin{abstract}
		The main purpose of this paper is the study of second-order optimality conditions for the bilinear control of a strongly degenerate parabolic equation. The equation is degenerate at the boundary of the spatial domain. The well-posedness of the state equation, as well as weak maximum principles are established. We prove some differentiability properties of the control-to-state operator and  the existence of optimal solutions. Finally, we derive first- and second-order optimality conditions for the system.
	\end{abstract}

	\maketitle

	\section{Introduction}
The main purpose of this paper is the study of the bilinear control problem for the following degenerate  parabolic equation
\begin{equation}\label{model}
\left\{
\begin{array}{rllll}
\dis y_{t}-\left(a(x)y_{x}\right)_{x} &=&v(t,x)\chi_{\omega}y& \mbox{in}& Q:=(0,T)\times \Omega,\\
\dis  a(x)y_x(\cdot,x)|_{x=\pm 1}&=&0& \mbox{in}& (0,T), \\
\dis  y(0,\cdot)&=&y^0 &\mbox{in}&\Omega,
\end{array}
\right.
\end{equation}
where $\Omega=(-1,1)$, $\omega$ is a nonempty subset of $\Omega$ and $\omega_T =(0,T )\times\omega $. In the problem \eqref{model}, $y=y(t,x)$ denotes the state and $v(t,x)$ is the control function acting on the system via the subset $\omega$.  Here $\chi_{\omega}$  denotes the characteristic function of the control set $\omega$ and $y^0\in L^2(\Omega)$ is the initial data.  We denote by $y_{t}$ and $y_{x}$ the partial derivative of $y$ with respect $t$ and $x$ respectively. The function $a:=a(\cdot)$ represents the diffusion coefficient and is assumed to vanish at the extreme points of $\overline{\Omega}$. The problem \eqref{model}  is said to be strongly degenerate when $a\in \mathcal{C}^1(\overline{\Omega})$ and so, $\dis\frac{1}{a}\notin L^1(\Omega)$.

Bilinear systems are used to describe complex phenomena in fields such as physics, chemistry, biology, economics and climate science \cite{epstein2013,ghil1976,north2017}. These types of systems model the interaction between large ice masses and solar radiation on climate \cite{budyko1969,sellers1969}. The state function in bilinear systems has a highly nonlinear dependence on the control function, which requires careful attention in obtaining results.

Recently, only a few studies have addressed bilinear optimal control problems and derived second-order necessary optimality conditions \cite{aronna2012, aronna2016,aronna2019,aronna2021ss, aronna2021s}. The authors in \cite{aronna2018} considered the optimal control of an infinite-dimensional bilinear system governed by a strongly continuous semigroup operator involving a time-dependent control in $L^1$ space. Using the Goh transform, they derived first and second optimality conditions.

Recently in \cite{aronna2021}, M. Aronna \textit{et al.} investigated a bilinear optimal control problem subject to the Fokker-Planck equation, deriving first and second order necessary and sufficient optimality conditions with a time-dependent control under mild regularity assumptions. However, many phenomena are described by degenerate equations, with degeneracy occurring at the boundary of the space domain, such as in sedimentation-consolidation processes \cite{tory2003} and biological models. In population genetics, degenerate systems are used in gene frequency models formulated as a Markov chain \cite{shimakura1992} or the Fleming-Viot model \cite{epstein2013}.

In this work, we study the bilinear optimal control of a strongly degenerate parabolic equation, where the control $v$ in \eqref{model} depends on time and space and acts on an open subset of the domain. We first establish a weak maximum principle to improve the regularity of solutions and prove the existence of optimal solutions, followed by deriving the first order optimality conditions. Then, using results from \cite{casas2012}, we obtain second order necessary and sufficient conditions for optimality. To the best of our knowledge, such results have not been studied for degenerate systems involving bilinear controls.

	We make the following assumptions:
	
	\begin{assumption}\label{ass1}
	\begin{equation}\label{as1}
 \begin{array}{llll}
			\dis a\in \mathcal{C}^1(\overline{\Omega}),\ \ a>0\ \mbox{in}\ \Omega \ \mbox{and}\ a(-1)=a(1)=0,\\
		 a \;\;\text{is such that}\,\,\,\dis \int_{0}^{x}\frac{ds}{a(s)}\in L^1(\Omega).
		\end{array}
	\end{equation}
\end{assumption}

	\begin{remark}\label{rmq1}$ $
An example of the diffusion coefficient $a$ is given by $a(x)=1-x^2$. In that case, the principal part of the operator in \eqref{model} coincides with that of the Budyko-Sellers model (see \cite{budyko1969, diaz2006, sellers1969}). It is an energy balance model studying the role played by continental and oceanic areas of ice on the evolution of the climate. They are diagnostic models intended to understand the evolution of the global climate on a long time scale. Their main characteristic is the high  sensitivity  to  the  variation  of  solar  and  terrestrial  parameters \cite{diaz1999}.   In dimension one, the unknown $y(t,x)$ represents the averaged temperature over each parallel, where $t$ denotes time and $x$ the sine of latitude. The degeneracy conditions on the boundary is justified by the fact that meridional heat flux at the poles must be zero \cite{diaz1999, north1975}. 
\end{remark}

Recently, some results have been obtained concerning the controllability for bilinear control problems applied on degenerate parabolic systems  (see e.g.  \cite{cannarsa2011a, floridia2014, floridia2020}). The authors in \cite{cannarsa2011} studied the global approximate multiplicative controllability for the linear degenerate parabolic Cauchy-Neumann problem \eqref{model}. The optimal control problems related to problem \eqref{model} have not yet been discussed in general. However, those bilinear optimal control problems remain very important from the perspective of their application in several fields. The main results of this paper are: we prove the existence and the uniqueness of the solution to the degenerate problem \eqref{model}; we establish maximum principle results for the problem \eqref{model} and finally, we derive the first and second conditions for optimality. Continuing the expansion of the findings presented in this paper, future research could explore analogous problems in higher-dimensional spaces, particularly on domains with distinctive geometries. Additionally, one could delve into a broad array of random differential equations or integro-differential models ( see e.g. \cite{arab2022, salim2023, tuncc2023}.)

 The paper is organized as follows. In Section \ref{wellposedness}, we give some notations and definitions of weighted functional spaces and their associated norms for the need of this work; we prove some results on existence and uniqueness of the weak solution of the problem \eqref{model} and establish some weak maximum principle results. In Section \ref{control}, we formulate the optimal control problem and prove the existence of an optimal control solution to \eqref{optimal}-\eqref{defJ}. Properties of the control-to-state mapping is discussed in Section \ref{controltostate}, and the derivation of the first order necessary optimality conditions is derived in Section \ref{firstorder}. Finally, we present the second order necessary and sufficient optimality conditions in Section \ref{sufficientoptcond}.

\section{Well-posedness results\label{wellposedness}}
Here, we discuss the well-posedness of problem \eqref{model} by introducing the weighted Sobolev spaces $H^1_a(\Omega)$ and $H^2_a(\Omega)$ defined as follows (loc. abs. cont. means locally absolutely continuous).
 \begin{equation*}
H^1_{a}(\Omega)=\{u\in L^2(\Omega): u\ \mbox{is loc. abs. cont. in}\ \Omega: \sqrt{a}u_x\in L^2(\Omega)\},
\end{equation*}
endowed with the norm  
$$
\|u\|^2_{H^1_{a}(\Omega)}=\|u\|^2_{L^2(\Omega)}+\|\sqrt{a}u_x\|^2_{L^2(\Omega)},\ \ \ u\in H^1_{a}(\Omega)
$$
and
\begin{equation*}\label{}
\dis H^2_{a}(\Omega)=\{u\in H^1_{a}(\Omega): a(x)u_x\in H^1(\Omega)\},
\end{equation*}
endowed with the norm
\begin{equation*}\label{}
\|u\|^2_{H^2_{a}(\Omega)}=\|u\|^2_{H^1_{a}(\Omega)}+\|(a(x)u_x)_x\|^2_{L^2(\Omega)}.
\end{equation*}

Under the Assumption \ref{ass1}, we have the following result (see \cite{cannarsa2011a}).
\begin{lemma}\label{embed}
We have
	\begin{equation*}
	H^1_a(\Omega)\hookrightarrow L^2(\Omega),
	\end{equation*}
	with compact embedding.
	\begin{remark}
		In contrast to the non-degenerate case, the functions in the Sobolev space $H^1_{a}(\Omega)$ are not necessarily bounded when the operator is strongly degenerate. Therefore, it is necessary to establish certain maximum principle results in order to ensure the boundedness of solutions to \eqref{model} (refer to Corollary \ref{theominmax}).
	\end{remark}

Let $(H^{1}_a(\Omega))'$ be the topological dual space of $H^1_a(\Omega)$ and we denote by $\left\langle \cdot,\cdot\right\rangle$ the dual product between $H^1_a(\Omega)$ and $(H^{1}_a(\Omega))'$.
If we set
\begin{equation}\label{defWTA}
W_a(0,T)= \left\{\rho \in L^2((0,T);H^1_a(\Omega)); \rho_t\in L^2\left((0,T);(H^{1}_a(\Omega))^\prime\right)\right\},
\end{equation}
then $W_a(0,T)$ endowed with the norm 
\begin{equation}\label{}
\|\rho\|^2_{	W_a(0,T)}=\|\rho\|^2_{L^2((0,T);H^1_a(\Omega))}+\|\rho_t\|^2_{L^2\left((0,T);(H^{1}_a(\Omega))^\prime\right)}
\end{equation}
is a Hilbert space. Moreover, we have that the following embedding is continuous 
\begin{equation}\label{contWTA}
W_a(0,T)\subset C([0,T],L^2(\Omega)). 
\end{equation}
 
\end{lemma}

\subsection{Existence results}\label{existence}

 this subsection, we will establish the existence and uniqueness of a solution to the problem \eqref{model}. Throughout the rest of the paper, the $L^\infty$-norm in $\omega_T$ will be represented by $\|\cdot\|_{\infty}$.
We first consider the following problem 
\begin{equation}\label{model1}
\left\{
\begin{array}{rllll}
\dis p_{t}-\left(a(x)p_{x}\right)_{x} &=&v(t,x)\chi_{\omega}p+f& \mbox{in}& Q,\\
\dis  a(x)p_x(\cdot,x)|_{x=\pm 1}&=&0& \mbox{in}& (0,T), \\
\dis  p(0,\cdot)&=&p^0 &\mbox{in}&\Omega.
\end{array}
\right.
\end{equation}
	Now, we define the weak solution to the problem \eqref{model1} as follows.
	\begin{definition}\label{weaksolution}
		Let $f\in L^2((0,T);(H^1_a(\Omega))')$, $v\in L^\infty(\omega_T)$ and $p^{0}\in L^2(\Omega)$. A function $p\in W_a(0,T)$ is a weak solution to \eqref{model1}, if the following equality 
		\begin{equation}\label{weaksol}
		\begin{array}{lll}
		\dis -\int_{0}^{T} \left\langle \phi_t,p \right\rangle\, \dt + \int_Q a(x)p_x\phi_x\, \dq-\int_{\omega_T}vp\, \phi \;\dq=
		\dis \int_\Omega p^0\,\phi(0,\cdot)\, \dx+\int_{0}^{T}\left\langle f, \phi\right\rangle \,\dt,
		\end{array}
		\end{equation}
		holds, for every
		\begin{equation}\label{hq}
			\phi \in H(Q)=\left\{z\in W_a(0,T):z(T,\cdot)=0 \hbox{ in } \Omega\right\}.
		\end{equation}
	\end{definition}
We state the well posedness of the problem \eqref{model1} in the following theorem.
	\begin{theorem}\label{theoremexistence0}
		Let $f\in L^2((0,T);(H^1_a(\Omega))')$, $v\in L^\infty(\omega_T)$ and $p^{0}\in L^2(\Omega)$.  Then, there exists a unique weak solution $p \in W_a(0,T)$ to \eqref{model1} in the sense of Definition~\ref{weaksolution}.
		In addition, the following estimate holds
		\begin{equation}\label{estimation0}
		\|p\|_{\C([0,T];L^2(\Omega))}+ \|p\|_{L^2((0,T);H^1_a(\Omega))} \leq 2e^{(\|v\|_{\infty}+1)T}\left(\|p^0\|_{L^2(\Omega)}+\|f\|_{L^2((0,T);(H^1_a(\Omega))')}\right).
		\end{equation}
		and 
		\begin{equation}\label{estimation00}
		\dis \|p\|_{W_a(0,T)} \leq 6(\|v\|_{\infty}+4)e^{2(\|v\|_{\infty}+1)T}\left(\|p^0\|_{L^2(\Omega)}+\|f\|_{L^2((0,T);(H^1_a(\Omega))')}\right).
		\end{equation}	
	\end{theorem}
	
Before the proof of Theorem \ref{theoremexistence0}, we state and prove the following intermediate result.\\
We set $z(t,x)=e^{-rt}p(t,x)$ for a suitable $r>0$, which will be chosen later. Then $p$ is solution to  \eqref{model1} if and only if $z$ is solution to

\begin{equation}\label{model2}
\left\{
\begin{array}{rllll}
\dis z_{t}-\left(a(x)z_{x}\right)_{x}+rz &=&v(t,x)\chi_{\omega}z+e^{-rt}f& \mbox{in}& Q,\\
\dis  a(x)z_x(\cdot,x)|_{x=\pm 1}&=&0& \mbox{in}& (0,T), \\
\dis  z(0,\cdot)&=&p^0 &\mbox{in}&\Omega.
\end{array}
\right.
\end{equation}	
	
We have the following definition.
\begin{definition}\label{weaksolution1}
	Let $f\in L^2((0,T);(H^1_a(\Omega))')$, $v\in L^\infty(\omega_T)$ and $p^{0}\in L^2(\Omega)$.  A function
	$z\in W_a(0,T)$ is a weak solution to \eqref{model2}, if for every $\phi\in H(Q)$ defined by \eqref{hq}, the following equality holds
	\begin{equation}\label{weaksol2}
	\begin{array}{lll}
	\dis -\int_{0}^{T} \left\langle \phi_t,p \right\rangle\, \dt + \int_Q a(x)z_x\phi_x\, \dq+r\int_{Q}z\phi\, \dq-\int_{\omega_T}vz\, \phi \;\dq=
	\dis \int_\Omega p^0\,\phi(0,\cdot) \dx+\int_{0}^{T}e^{-rt}\left\langle f, \phi\right\rangle \,\dt.
	\end{array}
	\end{equation}
\end{definition}	
Based on Theorem 1.1 from \cite{lions2013}, we can establish the well-posedness of the problem \eqref{model2}.
	\begin{theorem}\label{theoremexistence01}
	Let $f\in L^2((0,T);(H^1_a(\Omega))')$, $v\in L^\infty(\omega_T)$ and $p^{0}\in L^2(\Omega)$.  Then, the problem \eqref{model2} admits a unique weak solution $z \in W_a(0,T)$ in the sense of Definition~\ref{weaksolution1}.
	In addition, the following estimates hold
	\begin{equation}\label{estimation01}
	\|z\|_{\C([0,T];L^2(\Omega))}+ \|z\|_{L^2((0,T);H^1_a(\Omega))} \leq 2\left(\|p^0\|_{L^2(\Omega)}+\|f\|_{L^2((0,T);(H^1_a(\Omega))')}\right)
	\end{equation}
	and
	\begin{equation}\label{estimation02}
	\dis \|z\|_{W_a(0,T)} \leq 2(\|v\|_{\infty}+3)\left(\|p^0\|_{L^2(\Omega)}+\|f\|_{L^2((0,T);(H^1_a(\Omega))')}\right).
	\end{equation}	
\end{theorem}	
		
	\begin{proof}
We make the proof in three steps.\\
		\noindent \textbf{Step 1.}  For the existence result, we apply \cite[Theorem 1.1, Page 37 ]{lions2013}. We set $r=\|v\|_{\infty}+1$ and we endow $H(Q)$ with the norm defined by 
		$$
		\|\varphi\|^2_{H(Q)}:=\|\varphi\|^2_{L^2((0,T);H^1_a(\Omega))}+\|\varphi(0,\cdot)\|^2_{L^2(\Omega)},\, \forall \varphi \in H(Q).
		$$
		Therefore, it is clear that for any $\rho\in H(Q),$ we have that
		$$\|\rho\|_{L^2((0,T);H^1_a(\Omega))}\leq \|\rho\|_{H(Q)}.$$
		This shows that the embedding $H(Q)\hookrightarrow L^2((0,T);H^1_a(\Omega))$ is continuous.\par	
		Now, let $\varphi \in H(Q)$ and consider the bilinear form $\mathcal{E}(\cdot,\cdot)$ defined on $L^2((0,T);H^1_a(\Omega))\times H(Q)$ by:
		\begin{equation}\label{defCalE}
		\begin{array}{lll}
		\mathcal{E}(z,\phi)&=&\dis -\int_{0}^{T} \left\langle \phi_t,p \right\rangle\, \dt+ \int_Q a(x)z_x\phi_x\, \dq+r\int_{Q}z\phi\, \dq-\int_{\omega_T}vz\, \phi \;\dq.
		\end{array}
		\end{equation}
Using the Cauchy-Schwarz inequality, we obtain
\begin{equation*}
\begin{array}{lll}
|\mathcal{E}(z,\phi)|&\leq &\dis\left[[(2\|v\|_{\infty}+1)\|\phi\|_{L^2(Q)}+\|\phi_t\|_{L^2((0,T);(H^1_a(\Omega))')}]^2+\| \sqrt{a(\cdot)}\phi_x\|^2_{L^2(Q)}\right]^{1/2} \|z\|_{L^2((0,T);H^1_a(\Omega))}.
\end{array}
\end{equation*}

Therefore,  there exists a constant $C=C(\phi,\|v\|_{\infty})>0$  such that
		$$|\mathcal{E}(z,\phi)|\leq C\|z\|_{L^2((0,T);H^1_a(\Omega))}. $$
Consequently, for every fixed $\phi\in H(Q),$
		the functional  $z\mapsto \mathcal{E}(z,\phi)$ is continuous on $L^2((0,T);H^1_a(\Omega)).$\par 	
		Next, for every  $\phi\in H(Q)$ 
		$$\begin{array}{lllll}
		\mathcal{E}(\phi,\phi)&=&-\displaystyle\int_{0}^{T} \left\langle \phi_t,\phi \right\rangle\, \dt+\int_Q a(x)\phi^2_x\, \dq+r\int_{Q}\phi^2\, \dq-\int_{\omega_T}v\phi^2 \;\dq\\
		&\geq&\dis \frac 12  \|\phi(0,\cdot)\|^2_{L^2(\Omega)}+\int_Q a(x)\phi^2_x\, \dq+r\int_{Q}\phi^2\, \dq-\|v\|_{\infty}\int_{Q}\phi^2 \;\dq\\
	&\geq&	\dis \frac 12  \|\phi(0,\cdot)\|^2_{L^2(\Omega)}+\int_Q a(x)\phi^2_x\, \dq+\int_{Q}\phi^2\, \dq,\ \mbox{because}\ r=\|v\|_{\infty}+1\\
		&\geq &\dis \frac 12 \|\phi(0,\cdot)\|^2_{L^2(\Omega)}+ \|\phi\|^2_{L^2((0,T);H^1_a(\Omega))}\\
		&\geq &
		\dis \frac{1}{2}\|\phi\|^2_{H(Q)}.
		\end{array}
		$$	 		
Finally, we consider the linear functional $L(\cdot):H(Q)\to \R$ defined by
		$$
		L(\phi)=:\dis \int_{0}^{T}e^{-rt}\left\langle f, \phi\right\rangle \,\dt+\int_\Omega p^0\,\phi(0,\cdot)\ \dx.
		$$
	Then using the Cauchy-Schwarz inequality, we get 
	\begin{equation*}
	|L(\phi)|\leq \left(\|f\|^2_{L^2((0,T);(H^1_a(\Omega))')}+\|p^0\|^2_{L^2(\Omega)}\right)^{1/2}\|\phi\|_{H(Q)}.
	\end{equation*}
		Therefore, it follows from \cite[Theorem 1.1, page 37 ]{lions2013} that there exists $y\in L^2((0,T);H^1_a(\Omega))$ such that
		\begin{equation}\label{formvar}
		\mathcal{E}(y,\phi)= L(\phi),\quad \forall \phi \in H(Q).
		\end{equation}
		Hence, the problem \eqref{model2} has a solution $y\in L^2((0,T);H^1_a(\Omega))$ in the sense of Definition \ref{weaksolution1}. Moreover, since $z_{t} =\left(a(x)z_{x}\right)_{x}-rz+v(t,x)\chi_{\omega}z+e^{-rt}f\in L^2((0,T);(H^1_a(\Omega))'), $ we deduce that $z\in W_a(0,T)$.\par
		\noindent \textbf{Step 2.} We prove the uniqueness of the solution to \eqref{model2}.\\
		Assume that there exist $z_1$ and $z_2$ solutions to \eqref{model2} with the same right hand side $f,\, v$ and initial datum $p^0$.  Set $\tilde{z}:=z_1-z_2$. Then $\tilde{z}$ satisfies
		\begin{equation}\label{pazero}
		\left\{
		\begin{array}{rllll}
		\dis \tilde{z}_{t}-\left(a(x)\tilde{z}_{x}\right)_{x}+r\tilde{z} &=&v(t,x)\chi_{\omega}\tilde{z}& \mbox{in}& Q,\\
		\dis  a(x)\tilde{z}_x(\cdot,x)|_{x=\pm 1}&=&0& \mbox{in}& (0,T), \\
		\dis  \tilde{z}(0,\cdot)&=&0 &\mbox{in}&\Omega.
		\end{array}
		\right.
		\end{equation}
		If we multiply the first equation of \eqref{pazero} by $\tilde{z}$, then use an integration by parts over $Q$, we obtain
		$$
		 \frac 12\|\tilde{z}(T,\cdot)\|^2_{L^2(\Omega)}+\int_0^T \|\sqrt{a(\cdot)}\tilde{z}(t)\|^2_{L^2(\Omega)}\, dt+r\|\tilde{z}\|^2_{L^2(Q)}\leq \|v\|_{L^\infty(Q)}\|\tilde{z}\|^2_{L^2(Q)}.
		$$
		Using the fact that $r=\|v\|_{\infty}+1$ in this latter inequality,  we deduce that	
		$$\begin{array}{lll}	
		\dis \frac 12\|\tilde{z}(T,\cdot)\|^2_{L^2(\Omega)}+\|\tilde{z}\|^2_{L^2((0,T);H^1_a(\Omega))}&\leq&0.\\
		\end{array}$$
		Hence, $\tilde{z}=0$  in $\Omega$. Thus, $z_1=z_2$ in $\Omega$ and  shown the uniqueness.\\
			\noindent \textbf{Step 3.} We show the estimates \eqref{estimation01}.\\
			 First, we show that
			\begin{equation}\label{estcorpt}
			\dis \|z_t\|_{L^2((0,T);(H^1_a(\Omega))')} \leq 2(\|v\|_{\infty}+1)\|z\|_{L^2((0,T);H^1_a(\Omega))}+\|f\|_{L^2((0,T);(H^1_a(\Omega))')}.
			\end{equation}
			By applying the duality map between the first equation in \eqref{model2} and $\phi\in L^2((0,T);H^1_a(\Omega))$,
			we arrive to
			$$\begin{array}{rlll}
			\dis  \left\langle z_t(t),\phi(t)\right\rangle_{(H^1_a(\Omega))',H^1_a(\Omega)}&=&\dis - \int_\Omega a(x) z_x \phi_x\, \dx-\int_\Omega( r+v(t)\chi_{\omega})z \phi\, \dx +e^{-rt} \left\langle z(t),\phi(t)\right\rangle_{(H^1_a(\Omega))',H^1_a(\Omega)}.
			\end{array}
			$$
			This latter identity along with $r=\|v\|_{\infty}+1$ imply that
			$$
			\begin{array}{llll}
			\dis |\dis  \left\langle z_t(t),\phi(t)\right\rangle_{(H^1_a(\Omega))',H^1_a(\Omega)}| &\leq& \dis \dis \|\sqrt{a}z_x(t)\|_{L^2(\Omega)} \|\sqrt{a} \phi_x(t)\|_{L^2(\Omega)} + \left(2\|v\|_{\infty}+1\right)\|z(t)\|_{L^2(\Omega)}\|\phi(t)\|_{L^2(\Omega)}\\
			&&+\|f(t)\|_{(H^1_a(\Omega))'}\|\phi(t)\|_{H^1_a(\Omega)}
			.
			\end{array}
			$$
			Integrating this inequality over $(0,T)$,  we can deduce that
			\begin{equation} \label{estimationint3-1}
			\begin{array}{llll}
			\dis\int_0^T |\dis  \left\langle z_t(t),\phi(t)\right\rangle_{(H^1_a(\Omega))',H^1_a(\Omega)}|\, \dt\\
			\leq \left[ 2(\|v\|_{\infty}+1)\|z\|_{L^2((0,T);H^1_a(\Omega))}+\|f\|_{L^2((0,T);(H^1_a(\Omega))')}\right]
			\|\phi\|_{L^2((0,T);H^1_a(\Omega))},
			\end{array}
			\end{equation}
			from which we deduce \eqref{estcorpt}.
Next, multiplying the first equation in \eqref{model2}  by $z\in L^2((0,T);H^1_a(\Omega))$, using the integration by parts over $\Omega$, the Cauchy-Schwarz inequality and the Young inequality, we obtain
		$$
		\begin{array}{llll}
		\dis  \frac 12\frac{d}{dt}\|z(t)\|^2_{L^2(\Omega)}+\|\sqrt{a(\cdot)}z(t)\|^2_{L^2(\Omega)}+r\|z(t)\|^2_{L^2(\Om)}\\
		=\dis e^{-rt}\left\langle f(t), z(t)\right\rangle \,\dx+\int_{\omega} vz^2(t)\, \dx\\
		\leq \dis \frac{1}{2}\|f(t)\|^2_{(H^1_a(\Omega))'}+\frac{1}{2}\|z(t)\|^2_{H^1_a(\Omega)}+\dis \|v\|_{\infty}\|z(t)\|^2_{L^2(\Omega)}.
		\end{array}
		$$
		Hence, due to the fact that $r=\|v\|_{\infty}+1$, we obtain
		\begin{equation}\label{inter2}
		\begin{array}{llll}
		\dis \frac{d}{dt}\|z(t)\|^2_{L^2(\Omega)}+\|z(t)\|^2_{H^1_a(\Omega)}&\leq &\|f(t)\|^2_{(H^1_a(\Omega))'}.
		\end{array}
		\end{equation}
		Now,  integrating this latter inequality on $(0,\tau)$, with $\tau\in [0,T]$, we obtain
		$$\begin{array}{llll}
		\dis \|z(\tau)\|^2_{L^2(\Omega)}+\int_{0}^{\tau}\|z(t)\|^2_{H^1_a(\Omega)}\, \dt&\leq &\dis \int_{0}^{T}\|f(t,\cdot)\|^2_{(H^1_a(\Omega))'}\, \dt+\|p^0\|^2_{L^2(\Omega)},
		\end{array}
		$$
		from which we deduce that
		$$\begin{array}{lll}
		\dis\sup_{\tau\in [0,T]}\|z(\tau,\cdot)\|^2_{L^2(\Omega)}&\leq &\left[\|p^0\|^2_{L^2(\Omega)}+\|f\|^2_{L^2((0,T);(H^1_a(\Omega))')}\right],\\
		\dis \int_{0}^{T}\left\| z(t,\cdot)\right\|^2_{H^1_a(\Omega)}\, \dt &\leq&\dis \left[\|p^0\|^2_{L^2(\Omega)}+\|f\|^2_{L^2((0,T);(H^1_a(\Omega))')}\right].
		\end{array}
		$$
Hence, 
\begin{equation}\label{es2}
\begin{array}{lll}
\dis\|z\|_{\C([0,T];L^2(\Omega))}+\left\| z\right\|_{L^2((0,T);H^1_a(\Omega))}\leq 2\left(\|p^0\|_{L^2(\Omega)}+\|f\|_{L^2((0,T);(H^1_a(\Omega))')}\right).
\end{array}
\end{equation}
Combining \eqref{estcorpt} and \eqref{es2}, we deduce that
\begin{equation*}
\dis \|z\|_{W_a(0,T)} \leq 2(\|v\|_{\infty}+3)\left(\|p^0\|_{L^2(\Omega)}+\|f\|_{L^2((0,T);(H^1_a(\Omega))')}\right).
\end{equation*}	
This completes the proof.
	\end{proof}
Now, we prove the Theorem \ref{theoremexistence0}.
\begin{proof}[Proof of Theorem \ref{theoremexistence0}]
The existence of a unique weak solution $p \in W_a(0,T)$ to \eqref{model1} in the sense of Definition~\ref{weaksolution} follows directly from the Theorem \ref{theoremexistence01}. To prove the estimates \eqref{estimation0}-\eqref{estimation00}, we replace $z$ by its value $e^{-(\|v\|_{\infty}+1)t}p$ in  \eqref{estimation01} and \eqref{estimationint3-1} to obtain:
 \begin{equation}\label{44}
 \|p\|_{\C([0,T];L^2(\Omega))}+ \|p\|_{L^2((0,T);H^1_a(\Omega))} \leq 2e^{(\|v\|_{\infty}+1)T}\left(\|p^0\|_{L^2(\Omega)}+\|f\|_{L^2((0,T);(H^1_a(\Omega))')}\right).
 \end{equation}
and  
 $$
 \begin{array}{llll}
 \dis \int_0^T |\left\langle p_t(t),\phi(t)\right\rangle_{(H^1_a(\Omega))^\prime,H^1_a(\Omega)}| \dt\\
 \leq e^{(\|v\|_{\infty}+1)T} \left[ 3(\|v\|_{\infty}+1)\|p\|_{L^2((0,T);H^1_a(\Omega))}+\|f\|_{L^2((0,T);(H^1_a(\Omega))')}\right]
 \|\phi\|_{L^2((0,T);H^1_a(\Omega))},\,\forall \phi\in L^2((0,T);H^1_a(\Omega)),
 \end{array}
 $$
 which in view of  \eqref{44} implies that 
 $$
 \begin{array}{llll}
 \dis \|p_t\|_{L^2((0,T);(H^1_a(\Omega))^\prime)} \leq 6(\|v\|_{\infty}+2)e^{2(\|v\|_{\infty}+1)T}\left(\|p^0\|_{L^2(\Omega)}+\|f\|_{L^2((0,T);(H^1_a(\Omega))')}\right).
 \end{array}
 $$
 Finally, by combining this latter inequality with \eqref{44}, we arrive to \eqref{estimation00}.
\end{proof}

By taking $f=0$ in \eqref{model1}, we obtain the following well posed result for the problem \eqref{model}.

\begin{theorem}\label{theoremexistence}
	Let $v\in L^\infty(\Omega)$ and $y^{0}\in L^2(\Omega)$.  Then, there exists a unique weak solution $y \in W_a(0,T)$ to \eqref{model} in the sense of Definition~\ref{weaksolution} (with $f=0$).
	In addition, the following estimates hold
	\begin{equation}\label{estimation1}
	\|y\|_{\C([0,T];L^2(\Omega))}+ \|y\|_{L^2((0,T);H^1_a(\Omega))} \leq 2e^{(\|v\|_{\infty}+1)T}\|y^0\|_{L^2(\Omega)}
	\end{equation}
	and 
	\begin{equation}\label{estimation10}
	\dis \|y\|_{W_a(0,T)} \leq 6(\|v\|_{\infty}+4)e^{2(\|v\|_{\infty}+1)T}\|p^0\|_{L^2(\Omega)}.
	\end{equation}	
\end{theorem}
\begin{remark}\label{existencef'}
	Notice that if  $f\in L^2(Q)$, then the Theorem \ref{theoremexistence0} remain true and the estimates \eqref{estimation1} holds with $\|f\|_{L^2((0,T);(H^1_a(\Omega))')}$ replaced by $\|f\|_{L^2(Q)}$.
\end{remark}

	\subsection{Maximum principle results}
The objective of this subsection is to establish a weak maximum principle for the problem \eqref{model}, with the purpose of improving the regularity of the solution. We have the the following result.
	\begin{theorem}\label{positive}
		Let  $y^{0}\in L^2(\Omega)$  be  such that $y^0(x)\geq0$ almost everywhere in $\Omega$  and  $v\in L^\infty(\Omega)$. Then the weak solution of  \eqref{model} is positive almost everywhere in $Q$.
	\end{theorem}

	\begin{proof}
	Let $(t,x)\in Q$, we write $y(t,x)=y^+(t,x)-y^-(t,x)$, where $y^+(t,x)=\max(y(t,x),0)$ and $y^-(t,x)=\max(0,-y(t,x))$. It is sufficient to show that $y^-(t,x)=0$. Notice that 
		$$\begin{array}{llllll}
		y^+y^-=0 \hbox{ in } Q\;\; \text{and }\ \ y^-(0,x)=\max(0,-y^0(x))=0 \hbox{ in } \Omega,\\
		\end{array}$$
		and  $y^-\in L^2((0,T);H^1_a(\Omega)\cap \mathcal{C}([0,T];L^2(\Omega)).$ \par
	Multiplying the first equation in \eqref{model}  by $y^-(t,x)$ and integrating over $\Omega$, we arrive to
		\begin{equation}\label{es1}
			\begin{array}{llll}
		\dis\langle y_t (t),y^-(t)\rangle-\int_\Omega (a(x)y_x(t))_xy^-(t)\, \dx
		&=&\dis \int_{\omega}vy(t)y^-(t)\,\dx.
		\end{array}
		\end{equation}
		Recalling the definition of $y^+$ and $y^-$, we observe that
		\begin{equation*}
		\begin{array}{llll}
		\dis\langle y_t (t),y^-(t)\rangle=\dis\langle (y^+-y^-)_t (t),y^-(t)\rangle=-\dis\langle (y^-)_t (t),y^-(t)\rangle=-\dis \frac{1}{2}\frac{d}{dt}\|y^-(t) \|^2_{L^2(\Omega)}
		\end{array}
		\end{equation*}
		and 
		\begin{equation*}
		\begin{array}{llll}
		\dis-\int_\Omega (a(x)y_x(t))_xy^-(t)\, \dx= -[a(x)y_x(t)y^-(t)]^1_{-1}+\int_\Omega a(x)y_x(t)(-y_x)(t)\, \dx=-\int_\Omega a(x)y^2_x(t)\, \dx.
		\end{array}
		\end{equation*}
	Hence, \eqref{es1} can be rewritten as
		$$
		-\dis \frac{1}{2}\frac{d}{dt}\|y^-(t) \|^2_{L^2(\Omega)}-\int_\Omega a(x)y^2_x(t)\, \dx
		=-\dis \int_\omega v(t)\left|y^-(t)\right|^2\, \dx.
		$$
		From which we deduce that
		$$
		\dis \frac{d}{dt}\|y^-(t) \|^2_{L^2(\Omega)}\leq 2\|v\|_{\infty} \int_\Omega \left|y^-(t)\right|^2\, \dx.$$
		By using the Gronwall's lemma it follows that
		$$\|y^-(t) \|^2_{L^2(\Omega)}\leq e^{2T\|v\|_{\infty}}\|y^-(0,\cdot)\|_{L^2(\Omega)}=0.$$
		Hence 
		$y^-(t,x)=0$ for almost every $(t,x)\in Q$. Consequently,  $y\geq 0$ almost everywhere  in $Q$.
	\end{proof}
	The following result states a weak maximum principle for the problem \eqref{model1}.
	\begin{theorem}\label{theominmax0}
		Let  $p^0\in L^\infty(\Omega)$, $f\in L^\infty(Q)$ and $v\in L^\infty(\omega_T)$. Then  the weak solution $p\in W_a(0,T)$ of \eqref{model1} belongs to $L^\infty(Q)$ and the following estimate holds 
		\begin{equation}\label{minmax01}
		\begin{array}{lllll}
		 p(t,x) \leq  e^{(\|v\|_{\infty}+1)T} \left(\|p^0\|_{L^\infty(\Omega)}+\|f\|_{L^\infty(\Omega)}\right)\;\hbox{ a.e. in }  [0,T]\times\Omega.
		\end{array}
		\end{equation}
		 Moreover,
		\begin{equation}\label{minmax0}
		\begin{array}{lllll}
		\|p\|_{L^\infty(Q)} \leq  e^{(\|v\|_{\infty}+1)T} \left(\|p^0\|_{L^\infty(\Omega)}+\|f\|_{L^\infty(Q)}\right).
		\end{array}
		\end{equation}
	\end{theorem}
	
	\begin{proof} For any $(t,x)\in Q$, we set  $z(t,x)=e^{-(\|v\|_{\infty}+1)t}p(t,x)$, where $p$ is the weak solution to \eqref{model1}. Then according to Theorem \ref{theoremexistence01}, $z\in W_a(0,T)$ is the unique weak solution of \eqref{model2}, that we recall here
		\begin{equation*}
		\left\{
		\begin{array}{rllll}
		\dis z_{t}-\left(a(x)z_{x}\right)_{x}+(\|v\|_{\infty}+1)z &=&v(t,x)\chi_{\omega}z+e^{-(\|v\|_{\infty}+1)t}f& \mbox{in}& Q,\\
		\dis  a(x)z_x(\cdot,x)|_{x=\pm 1}&=&0& \mbox{in}& (0,T), \\
		\dis  z(0,\cdot)&=&p^0 &\mbox{in}&\Omega.
		\end{array}
		\right.
		\end{equation*}		
We claim that 
\begin{equation}\label{ineqint}
 z(t,x)\leq \|p^0\|_{L^\infty(\Omega)}+\|f\|_{L^\infty(Q)}\;\hbox{ a.e. in }  [0,T]\times\Omega.
\end{equation}		
We set $w(t,x)=\|p^0\|_{L^\infty(\Omega)}+\|f\|_{L^\infty(Q)}-z(t,x)$ for any $(t,x)\in (0,T)\times\Omega$. Then,
		$$\begin{array}{lll}
		w(0,x)&=&\|f\|_{L^\infty(Q)}+(\|p^0\|_{L^\infty(\Omega)}-p^0(x)),\\
		&\geq& 0,
		\end{array}
		$$
for any $x\in \Omega.$ Moreover, $w$ satisfies
		\begin{equation}\label{modelwb}
		\left\{
		\begin{array}{rllll}
		\dis  w_t-(a(x)w_x)_x+ (\|v\|_{\infty}+1)w&=& vw\chi_{\omega }+ ( \|v\|_{\infty}-v\chi_{\omega })K+\left(K-e^{-(\|v\|_{\infty}+1)t}f\right)\qquad &\mbox{in}& Q,\\
		\dis  a(x)w_x(\cdot,x)|_{x=\pm 1}&=&0  &\mbox{in}& (0,T) ,\\
		w(0,\cdot)&=& \|p^0\|_{L^\infty(\Omega)}-p^0&\mbox{in}& \Omega,
		\end{array}
		\right.
		\end{equation}
		with $K=\|p^0\|_{L^\infty(\Omega)}+\|f\|_{L^\infty(Q)}$.
We write $w(t,x)=w^+(t,x)-w^-(t,x)$, where $w^+(t,x)=\max(w(t,x),0)$ and $w^-(t,x)=\max(0,-w(t,x))$. It is sufficient to show that $w^-(t,x)=0$ for almost every $(t,x)\in Q$. That is $w^-\in L^2((0,T);H^1_a(\Omega)\cap \mathcal{C}([0,T];L^2(\Omega)).$  Notice that $w^+w^-=0$ and since $w(0,x)\geq 0$ in $\Omega$,
$w^-(0,x)=\max(0,-w(0,\cdot))=0 \hbox{ in } \Omega.$  \par

If we multiply the first equation in \eqref{modelwb}  by $w^-(t)$ and integrate over $\Omega$, we obtain
\begin{equation}\label{aj1}\begin{array}{llll}
		\dis\langle w_t (t),w^-(t)\rangle-\int_\Omega (a(x)w_x(t))_xw^-(t)\, \dx +(\|v\|_{\infty}+1)\int_\Omega w(t)w^-(t)\,\dx \\
		=\dis \int_{\omega}v(t)w(t)w^-(t)\,\dx+K\int_{\omega}( \|v\|_{\infty}-v(t))w^-(t)\,\dx+\int_{\Omega}\left(K-e^{-(\|v\|_{\infty}+1)t}f(t)\right)w^-(t)\,\dx.
		\end{array}
		\end{equation}
	Observe that		
		\begin{equation*}
		\begin{array}{llll}
		\dis\langle w_t (t),w^-(t)\rangle=- \langle w^-_t (t),w^-(t)\rangle=\dis -\frac{1}{2}\frac{d}{dt}\|w^-(t)\|^2_{L^2(\Omega)},
		\end{array}
         \end{equation*}
        
        \begin{equation*}
        \begin{array}{llll}
        \dis-\int_\Omega (a(x)w_x(t))_xw^-(t)\, \dx= \int_\Omega a(x)w_x(t)(-w_x(t)) \,\dx=-\int_\Omega a(x)w^2_x(t)\, \dx,
        \end{array}
        \end{equation*}
        
        \begin{equation*}
        \begin{array}{llll}
        \dis(\|v\|_{\infty}+1)\int_\Omega w(t)w^-(t)\,\dx=- (\|v\|_{\infty}+1)\int_\Omega (w^-(t))^2\,\dx
        \end{array}
        \end{equation*}		
and 		
		
		\begin{equation*}
		\begin{array}{llll}
		\dis\int_\Omega v(t) w(t)w^-(t)\,\dx=- \int_\Omega v(t)(w^-(t))^2\,\dx.
		\end{array}
		\end{equation*}	
	Therefore, \eqref{aj1} can be rewritten as
	\begin{equation}\label{aj2}\begin{array}{llll}
	\dis\frac{1}{2}\frac{d}{dt}\|w^-(t)\|^2_{L^2(\Omega)}+ \int_\Omega a(x)w^2_x(t)\, \dx+(\|v\|_{\infty}+1)\int_\Omega (w^-(t))^2\,\dx\\
	=\dis \int_\Omega v(t)(w^-(t))^2\,\dx -K\int_{\omega}( \|v\|_{\infty}-v(t))w^-(t)\,\dx-\int_{\Omega}\left(K-e^{-(\|v\|_{\infty}+1)t}f(t)\right)w^-(t)\,\dx.
	\end{array}
	\end{equation}
On the one hand, since $a(\cdot)>0$ in $\Omega$, then the left hand side of \eqref{aj2} can be lower bounded by its first term.  On the other hand, 	$\|v\|_{\infty}-v\geq 0$ and $\dis K-e^{-(\|v\|_{\infty}+1)t}f=\|p^0\|_{L^\infty(\Omega)}+\left(\|f\|_{L^\infty(Q)}-e^{-(\|v\|_{\infty}+1)t}f\right)\geq  0$ almost every where in $Q$ and so the right hand side of \eqref{aj2} can be upper bounded by the term $\dis \|v\|_{\infty} \int_\Omega \left(w^-(t)\right)^2 dx$. Hence from \eqref{aj2}, we deduce that

$$\dis \frac{d}{dt}\|w^-(t) \|^2_{L^2(\Omega)}\leq 2\|v\|_{\infty} \int_\Omega \left(w^-(t)\right)^2 \,\dx.$$
		By using the Gronwall's lemma it follows that
		$$\|w^-(t) \|^2_{L^2(\Omega)}\leq e^{2t\|v\|_{\infty}}\|w^-(0,\cdot)\|_{L^2(\Omega)}=0.$$
		Hence 
		$w^-(t,x)=0$ for almost every $(t,x)\in Q$, that is $w(t,x)\geq0$ for almost every $(t,x)\in Q$. Consequently, 
		 \begin{equation*}
	z(t,x)\leq \|p^0\|_{L^\infty(\Omega)}+\|f\|_{L^\infty(Q)}\;\hbox{ a.e. in } [0,T]\times\Omega.
		\end{equation*}	
		
Now by taking $z(t,x)=e^{-(\|v\|_{\infty}+1)t}p(t,x)$ in this latter inequality, we deduce \eqref{minmax01} and then \eqref{minmax0}. 
	\end{proof}


By taking $f=0$ in the Theorem \ref{theominmax0}, we obtain the following weak maximum principle for the problem \eqref{model}.

\begin{corollary}\label{theominmax}
	Let  $y^0\in L^\infty(\Omega)$ and $v\in L^\infty(\omega_T)$. Then  the weak solution $y\in W_a(0,T)$ of \eqref{model} belongs to $L^\infty(Q)$. More precisely, the following estimate holds 
	\begin{equation}\label{minmax}
	\begin{array}{lllll}
	y(t,x) \leq  e^{(\|v\|_{\infty}+1)T}\|y^0\|_{L^\infty(\Omega)}\;\hbox{ a.e. in } [0,T]\times\Omega.
	\end{array}
	\end{equation}
 Moreover,
	\begin{equation}\label{minmaxy}
	\begin{array}{lllll}
	\|y\|_{L^\infty(Q)} \leq  e^{(\|v\|_{\infty}+1)T} \|y^0\|_{L^\infty(\Omega)}.
	\end{array}
	\end{equation}
\end{corollary}

\section{Resolution of the optimization problem \label{control}}
We state the following definition.
	\begin{definition}\label{ctso}
		We define the control-to-state mapping
		$$
		G:L^\infty(\omega_T)\to W_a(0,T),
		$$
		which associates to each $u\in L^\infty(\omega_T)$ the unique weak solution $y\in W_a(0,T)$ of \eqref{model}. Sometimes, we may write $G(u)$ to denote the state $y$ corresponding to the control $u$.
	\end{definition}
	We consider the following optimal control problem:
	\begin{equation}\label{optimal}
	\inf_{v\in \U} J(G(v),v),
	\end{equation}
	where 
	\begin{equation}\label{defuad}
	\dis \U=\left\{w\in L^\infty(\omega_T): m\leq w\leq M\right\}
	\end{equation}
	and
	\begin{equation}\label{defJ}
	J(G(v),v)=\frac{1}{2}\|G(v)(T)-y^d\|^2_{L^2(\Omega)}+ \frac{\alpha}{2}\|v\|^2_{L^2(\omega_T)},
	\end{equation}
	with $ \alpha>0$, $m,M\in \R$ with $m<M$, $y^d\in L^\infty({\Omega})$ and $G(v) \in W_a(0,T)\cap L^\infty(Q)$ is solution to \eqref{model}.\par
	\begin{definition}\label{defopt}
		\item We say that $u\in \U$ is a $L^\infty$-local solution (resp., $L^2$-local solution) of \eqref{optimal}-\eqref{defJ} if there exists $\varepsilon>0$ such that $J(G(u),u)\leq J(G(v),v)$ holds for every $v\in \U\cap B^\infty_\varepsilon(u)$ (resp. $v\in \U\cap B^2_\varepsilon(u)$), where $ B^\infty_\varepsilon(u)$ (resp., $B^2_\varepsilon(u)$) denote the open ball in  $L^\infty(Q)$ (resp., $L^2(Q)$) of radius $\varepsilon$ centered at $u$.
	\end{definition}
	\subsection{Existence of optimal controls}
	In this Section we prove the existence of optimal controls.	
	\begin{theorem}\label{existcontrol1}
		Let  $\alpha>0$, $v\in \mathcal{U}$ and $y^{0},y^d\in L^\infty(\Omega)$. Then, there exists at least one solution $u\in \mathcal{U}$ to the optimal control problem \eqref{optimal}-\eqref{defJ}.
	\end{theorem}
	\begin{proof}
		We observe that $J(G(v),v)\geq 0$ for all $v\in \mathcal{U}$. Now, let $\{(G(v^n),v_n)\}_n\subset W_a(0,T)\times\mathcal{U}$ be a minimizing sequence such that
		$$\lim_{n\to +\infty}J(G(v^n),v^n)= \inf_{v\in\U}J(G(v),v). $$
		Then, there is a positive constant $C$ independent of $n$ which may varies form one line to another such that
		\begin{equation}\label{bound01}
		\|y^n(T)\|_{L^2(\Omega)}\leq C	
		\end{equation}
		and 
		\begin{equation}\label{bound11}
		\|v^n\|_{L^2(\omega_T)}\leq C.
		\end{equation}
		Since $\|v_n\|_{\infty}\leq \max\{|m|,|M|\}$ and  $y^n$ being solution of \eqref{model} with $v=v^n$, we can deduce from \eqref{estimation10} that there is a positive constant $C$ independent of $n$ such that
		\begin{equation}\label{b2}	
		\|y^n\|_{W_a(0,T)} \leq C.
		\end{equation}
		Moreover, from \eqref{estimation0} we can write 
		\begin{equation*}
		\|y^n\|_{L^2((0,T);H^1_a(\Omega))} \leq 2e^{(\|v^n\|_{\infty}+1)T}\|y^0\|_{L^2(\Omega)}
		\leq 2e^{(\beta+1)T}\|y^0\|_{L^2(\Omega)},
		\end{equation*}
		with	$\beta=\max\{|m|,|M|\}$.
		Consequently, 
		\begin{eqnarray}\label{es3}
		\|v^ny^n\|_{L^2(\omega_T )}\leq C\|y^n\|_{L^2(Q)}\leq C\|y^n\|_{L^2((0,T);H^1_a(\Omega))}
		\leq 2Ce^{(\beta+1)T}\|y^0\|_{L^2(\Omega)}.
		\end{eqnarray}
		From the boundedness of $\U$ in $L^\infty(\omega_T)$ and \eqref{bound01}-\eqref{es3}, we deduce the existence of  $u\in L^\infty(\omega_T)$, $\eta\in L^2(\Omega)$, $y\in W_a(0,T) $ and $\zeta\in L^2(\omega_T)$ such that up to a subsequence and as $n\to +\infty$, 
		\begin{equation}\label{c1}
		v^n\rightharpoonup u   \text{ weakly-$\star$ in } L^\infty(\omega_T),
		\end{equation}	
		\begin{equation}\label{c1a}
		v^n\rightharpoonup u   \text{ weakly in } L^2(\omega_T),
		\end{equation}
		\begin{equation}\label{c1b}
		y^n(T)\rightharpoonup \eta   \text{ weakly in } L^2(\Omega),
		\end{equation}
		\begin{equation}\label{c2}
		y^n\rightharpoonup y   \text{ weakly in }  W_a(0,T),
		\end{equation}
		and 
		\begin{equation}\label{c41}
				v^ny^n\rightharpoonup \zeta    \text{ weakly in }   L^2(\omega_T).
				\end{equation}
		Using the compact embedding $W_a(0,T)\hookrightarrow L^2(Q)$, we deduce from \eqref{c2} that
		\begin{equation}\label{c3}
		y^n\to y   \text{ strongly in }  L^2(Q).
		\end{equation}
Moreover, 	
	Taking \eqref{c1a} and \eqref{c3} into account and using the weak-strong convergence, we get that,  as $n\to+\infty$
	\begin{equation}\label{c61}
	v^ny^n\rightharpoonup uy    \text{ weakly in } L^1(\omega_T),
	\end{equation}
	which in view of the continuous embedding $L^2(\omega_T)\hookrightarrow L^1(\omega_T)$,  \eqref{c41} and the uniqueness of the weak limit,  imply that  $\zeta=uy$. We have shown that,  as $n\to +\infty$
		\begin{equation}\label{es4}
	v^ny^n\rightharpoonup uy   \text{ weakly in }  L^2(\omega_T).
	\end{equation}	
		
%
%
%
%
%
%
%
		Since $\U$ is a closed convex subset of $L^2(\omega_T)$,  then $\U$ is weakly closed and so
		\begin{equation}\label{ajout5}
		u\in \U.
		\end{equation}
Now, we show that $y=y(u)$ and	
\begin{equation}\label{convroT}
y^n(T)\rightharpoonup y(T)  \text{ weakly in }   L^2(\Omega).
\end{equation}	
We use similar arguments as in \cite{dorville2021,kenne2020,cb2021,cprg2021,landry2020}.
We first recall that $y^n$ is satisfies
	\begin{equation}\label{pn1}
	\left\{
	\begin{array}{rllll}
	\dis (y^n)_{t}-\left(a(x)(y^n)_{x}\right)_{x} &=&v^n(t,x)\chi_{\omega}y^n& \mbox{in}& Q,\\
	\dis  a(x)y^n_x(\cdot,x)|_{x=\pm 1}&=&0& \mbox{in}& (0,T), \\
	\dis  y^n(0,\cdot)&=&y^0 &\mbox{in}&\Omega.
	\end{array}
	\right.
	\end{equation}
Since $y^n\in W_a(0,T)$ is a weak solution of \eqref{pn1}  in the sense of Definition \ref{weaksolution}, then we have
\begin{equation}\label{p01}
\begin{array}{lll}
\dis -\int_Q \phi_ty^n\, \dq + \dis \int_{Q}a(x)y^n_x\phi\,\dq=
\dis \int_{\omega_T}v^ny^n\, \phi \;\dq +\dis \int_\Omega y^0\,\phi(0) \,\dx,
\end{array}
\end{equation}		
for any $\phi\in H(Q)$.
Now, passing to the limit in \eqref{p01} as $n\to +\infty$, while using \eqref{c2} and \eqref{es4}, we obtain	
\begin{equation*}
\begin{array}{lll}
\dis -\int_Q \phi_ty \,\dx\, dt + \dis \int_{Q}a(x)y_x\phi\,\dq= 
\dis \int_{\omega_T}uy\, \phi \;\dq+\int_\Omega y^0\,\phi(0)\, \dx,
\end{array}
\end{equation*}	
for any $\phi\in H(Q)$. Hence, we deduce that $y=y(u)$ is a weak solution to \eqref{model} with $v=u$ and $f=0$ in the sense of Definition \ref{weaksolution}. 
Now let $\phi \in L^2((0,T);H^1_a(\Omega))\cap H^1((0,T);L^2(\Omega))$. If we multiply the first equation of \eqref{pn1} by $\phi$ and we integrate by parts over $Q$, we obtain 
	\begin{equation}\label{p010}
	\begin{array}{lll}
	\dis \int_{\Om}y^n(T)\phi(T)\,\dx-\int_Q \phi_ty^n\, \dq + \dis \int_{Q}a(x)y^n_x\phi\,\dq=
	\dis \int_{\omega_T}v^ny^n\, \phi \;\dq
	+\dis \int_\Omega \rho^0\,\phi(0)\, \dx.
	\end{array}
	\end{equation}	
Now, passing to the limit in \eqref{p010} as $n\to +\infty$, while using \eqref{c1b}, \eqref{c2} and \eqref{es4}, we obtain	
	\begin{equation*}
	\begin{array}{lll}
	\dis \int_{\Om}\eta \phi(T) \,\dx -\int_Q \phi_ty \,\dx\, dt + \dis \int_{Q}a(x)y_x\phi\,\dq=
	\dis \int_{\omega_T}uy\, \phi \;\dq
	+\dis \int_\Omega \rho^0\,\phi(0)\, \dx,
	\end{array}
	\end{equation*}
which by using again the integration by parts over $Q$,
	\begin{equation*}
	\begin{array}{lll}
	\dis \int_\Omega y^0\,\phi(0) \,\dx +\int_{\Om}\phi(T)(\eta-y(T))\ \dx+\int_Q \phi(y_t-(a(x)y_x)_x) \dq = \\
	\dis \int_{\omega_T}uy\, \phi \;\dq+\int_\Omega y^0\,\phi(0)\, \dx,\ \ 
	\forall \phi \in L^2((0,T);H^1_a(\Omega))\cap \mathcal{C}([0,T];L^2(\Omega)).
	\end{array}
	\end{equation*}
	Since  $y=y(u)$ is solution of \eqref{model}, we get from this latter identity that
	\begin{eqnarray}\label{interm11}
	\int_{\Om}\phi(T)(\eta-y(T))\, \dx=0, \quad
	\forall \phi \in L^2((0,T);H^1_a(\Omega))\cap H^1((0,T);L^2(\Omega)).
	\end{eqnarray}		
	Hence,  we deduce that
	\begin{eqnarray}\label{limp51}
	\eta=y(T)\quad \text{in}\quad \Om.
	\end{eqnarray}
	Combining \eqref{c1b} and \eqref{limp51}, we obtain \eqref{convroT}.
		Moreover, using \eqref{c1a}, \eqref{ajout5}, \eqref{convroT}  and the lower semi-continuity of the functional cost $J$, it follows that
		\begin{eqnarray*}
			J(G(u),u)&\leq &\liminf_{n\to+ \infty}J(G(v^n),v^n)=\inf_{v\in\U}J(G(v),v)
		\end{eqnarray*}	
	and the proof is complete. 		
	\end{proof}

	\section{Regularity results of the control-to-state mapping}
	\label{controltostate}
	In the rest of the paper, we assume that $y^0, y^d\in L^\infty(\Omega)$, so the weak solution $y$ of \eqref{model} belongs to $W_a(0,T)\cap L^\infty(Q)$ (see Corollary \ref{theominmax}). 
	
	Using the Implicit Function Theorem, we establish in this section some regularity results on the control-to-state operator. The proofs are inspired from \cite{aronna2021}. We define the mapping 
	\begin{equation}\label{defG}
	(y, u)\mapsto\mathcal{G}(y,u):= (y_t-(a(x)y_x)_x-u\chi_{\omega }y,y(0)-y^0)
	\end{equation}
	from $ W_a(0,T)\times L^\infty(\omega_T)\to L^2((0,T);(H^1_a(\Omega))')\times L^2(\Omega).$
	Then, the state equation $y$ solution of \eqref{model} can be viewed as the equation
	\begin{equation}\label{eqG}
	\mathcal{G}(y,u)=0.
	\end{equation}
	
	\begin{proposition}\label{lemmeG}
		The mapping $\mathcal{G}$ is of class $\mathcal{C}^{\infty}$. Moreover, the control-to-state mapping $G:L^\infty(\omega_T)\to W_a(0,T), u\mapsto y$ is also of class $\mathcal{C}^{\infty}$.
	\end{proposition}
	\begin{proof}
		For all $\phi\in L^2((0,T);H^1_a(\Omega))$, we can write the first component $\mathcal{G}_1$ of $\mathcal{G}$ as
		$$\mathcal{G}_1(y,u)(\phi):=\left\langle y_t,\phi\right\rangle+ \int_Q a(x)y_x\phi_x\, \dq-\int_{\omega_T}uy\, \phi \;\dq.$$ On the one  hand, the first two terms of $\mathcal{G}_1$ define linear and continuous mapping from $ W_a(0,T)$ to\\ $L^2((0,T);(H^1_a(\Omega))')$. Thus they are of class $\mathcal{C}^{\infty}$. On the other hand the last term of $\mathcal{G}_1$ defines a bilinear and continuous mapping from $ W_a(0,T)\times L^\infty(\omega_T)$ to $L^2((0,T);(H^1_a(\Omega))').$ Therefore $\mathcal{G}_1$ is of class $\mathcal{C}^{\infty}$. The second component of $\mathcal{G}$ is clearly of class $\mathcal{C}^{\infty}$. 
		
		Moreover,
		$$\partial_y\mathcal{G}(y,u)\varphi=(\varphi_t-(a(x)\varphi_x)_x-u\varphi\chi_{\omega },\varphi(0)).$$
		For $u\in L^\infty(\omega_T)$,  $\varphi^0\in L^2(\Omega)$ and $f\in L^2((0,T);(H^1_a(\Omega))')$,  the Theorem \ref{theoremexistence0} shows that the problem
		\begin{equation*}
		\left\{
		\begin{array}{rllll}
		\dis  \varphi_t-(a(x)\varphi_x)_x &=& u\varphi\chi_{\omega } +f\qquad &\mbox{in}& Q,\\
		\dis a(x)\varphi_x(\cdot,x)|_{x=\pm 1}&=&0  &\mbox{in}& (0,T) ,\\
		\varphi(0,\cdot)&=& \varphi^0 &\mbox{in}& \Omega
		\end{array}
		\right.
		\end{equation*}
		has a unique weak solution $\varphi:=\varphi(\varphi^0,f)$ in $W_a(0,T)$. Moreover, $\varphi$ depends continuously on $\varphi^0\in L^2(\Omega)$ and on $f\in L^2((0,T);(H^1_a(\Omega))')$. Consequently, the operator $\dis \partial_y\mathcal{G}(y,u)$ defines an isomorphism from $W_a(0,T)$ to $L^2((0,T);(H^1_a(\Omega))')\times L^2(\Omega)$. Using the Implicit Function Theorem, we deduce that $\mathcal{G}(y,u)=0$ implicitly defines the control-to-state operator $G:u\mapsto y$ which is itself of class  $\mathcal{C}^{\infty}$. 
	\end{proof}

	
The following result establishes the Lipschitz continuity of the control-to-state mapping $G$.
	\begin{proposition}\label{prop2} Let  $v\in L^\infty(\omega_T)$ and $y^{0}\in L^\infty(\Omega)$. Then the control-to-state mapping  $v\mapsto G(v)$ is a locally Lipschitz continuous function from $L^2(\omega_T)$ into $W_a(0,T)$. More precisely, for all $v_1,v_2\in L^\infty(\omega_T)$, there is a constant $C=C(\|v_1\|_{\infty},\|v_2\|_{\infty},T)>0$ such that the following estimate holds
		\begin{equation}\label{estim20}
	\|G(v_1)-G(v_2)\|_{\mathcal{C}([0,T];L^2(\Omega))}+	\|G(v_1)-G(v_2)\|_{L^2((0,T);H^1_a(\Omega))}\leq C \|y^0\|_{L^\infty(\Omega)}\|v_1-v_2\|_{L^2(\omega_T)}.
		\end{equation}	
	\end{proposition}
	\begin{proof}
		Let $v_1,v_2\in L^\infty(\omega_T)$ and  $z:=y(v_1)-y(v_2)$, where $y(v_1)$ and $y(v_2)$ are solutions of \eqref{model} with $v=v_1$ and $v=v_2$, respectively. Then, $z$ satisfies the following problem:
		\begin{equation}\label{a1}
		\left\{
		\begin{array}{lllll}
		\dis  z_t-(a(x)z_x)_x&=&(v_1z+(v_1-v_2)y(v_2))\chi_{\omega } &\hbox{in} & Q,\\
		\dis a(x)z_x(\cdot,x)|_{x=\pm 1}&=&0  &\mbox{in}& (0,T),\\
		z(0,\cdot)&=&0 &\hbox{in}& \Omega.
		\end{array}
		\right.
		\end{equation}
	Using the estimation \eqref{estimation0} with $f=(v_1-v_2)y(v_2)\chi_{\omega }\in L^2(Q)$, it follows that
		\begin{equation}\label{ajout0}
		\|z\|_{\mathcal{C}([0,T];L^2(\Omega))}+ \|z\|_{L^2((0,T);H^1_a(\Omega))} \leq 2C\|G(v_2)\|_{L^\infty(Q)}\|v_1-v_2\|_{L^2(\omega_T)},
		\end{equation}
		for some $C=C(\|v_1\|_{\infty},T)>0$.
		Since $y^0\in L^\infty(\Omega)$, we can apply Theorem \ref{theominmax}. Consequently, using \eqref{minmaxy}, we obtain:
		\begin{equation}\label{ajout4}
		\|G(v_2)\|_{L^\infty(Q)}\leq C(\|v_2\|_{\infty},T) \|y^0\|_{L^\infty(\Omega)}.
		\end{equation}
		Therefore, by combining \eqref{ajout0} and \eqref{ajout4}, we deduce that
		\begin{eqnarray*}
			\|z\|_{\C([0,T];L^2(\Omega))}+ \|z\|_{L^2((0,T);H^1_a(\Omega))} \leq C(\|v_1\|_{\infty},\|v_2\|_{\infty},T) \|y^0\|_{L^\infty(\Omega)}\|v_1-v_2\|_{L^2(\omega_T)},
		\end{eqnarray*}
	from which we deduce \eqref{estim20}.
	\end{proof}
	

	\section{First-order necessary optimality conditions}\label{firstorder}

	Before going further, we define the reduced cost functional as follows
	\begin{equation}\label{defJ1}
	\mathcal{J}(v):=J(G(v),v).
	\end{equation} 
It is worth noting that $\mathcal{J}$ is continuously Fr\'echet differentiable due to the fact that both $J$ and $G$ possess this property. 
The main purpose of this section is to establish the following result.

\begin{theorem}[First order necessary optimality conditions]\label{theoSO}
	Let $\alpha>0,$  $v\in \mathcal{U}$ and $y^{0},y^d\in L^\infty(\Omega)$. Let $u$ be an $L^\infty$-local minimum for the minimization problem \eqref{optimal}. Then the following identity holds
	\begin{equation}\label{ineq}
	\mathcal{J}'(u)(v-u)\geq 0\;\;\;\text{for every}\;\;\; v\in \U.
	\end{equation}
	
Moreover, there exist $y, q\in W_a(0,T)\cap L^\infty(Q)$ such that the triple $(y,u,q)$ satisfies	
	\begin{equation}\label{ad11}
	\left\{
	\begin{array}{lllll}
	\dis  y_t-(a(x)y_x)_x&=&uy\chi_{\omega } &in & Q,\\
	\dis a(x)y_x(\cdot,x)|_{x=\pm 1}&=&0  &\mbox{in}& (0,T),\\
	y(0,\cdot)&=&y^0 &in& \Omega,
	\end{array}
	\right.
	\end{equation}	
	\begin{equation}\label{ad21}
	\left\{
	\begin{array}{lllll}
	\dis  -q_t-(a(x)q_x)_x&=&uq\chi_{\omega } &in & Q,\\
	\dis a(x)q_x(\cdot,x)|_{x=\pm 1}&=&0  &\mbox{in}& (0,T),\\
	q(T,\cdot)&=&y(T;u)-y^d&in& \Omega,
	\end{array}
	\right.
	\end{equation}
	
	\begin{eqnarray}\label{neccoptcond}
	\int_{\omega_T}(\alpha u+y(u)q)(v-u)\ \dq\geq 0, \quad \forall v\in \mathcal{U},
	\end{eqnarray}
	and equivalently
	\begin{equation}\label{contr1}
		u=\min\left(\max\left(m,-\frac{q}{\alpha}y(u)\right),M\right) \text{ in }\omega_T.
	\end{equation}
	
\end{theorem}
	
%
Before proving the previous result, we establish some important results needed further.
The proofs follow similar arguments as in \cite{cyrillek2022}. We just provide a brief overview of the proofs.
	\begin{proposition}\label{diff0}
		Let $G:L^\infty(\omega_T)\to W_a(0,T)\cap L^\infty(Q), v\mapsto y$ be the control-to-state operator,  where $y$ is the solution of \eqref{model}. Then, the directional derivate of $G$ in every direction $w\in L^\infty(\omega_T)$ is given by
		$$G'(u)w=\rho,$$
		where the state $y=G(u)$ corresponds to $u$ and $\rho\in W_a(0,T)$ is the unique weak solution of
		\begin{equation}\label{diff1}
		\left\{
		\begin{array}{lllll}
		\dis  \rho_t-(a(x)\rho_x)_x &=&(u\rho+wy)\chi_{\omega } &\hbox{in} & Q,\\
		\dis a(x)\rho_x(\cdot,x)|_{x=\pm 1}&=&0  &\mbox{in}& (0,T),\\
		\rho(0,\cdot)&=&0 &\hbox{in}& \Omega.
		\end{array}
		\right.
		\end{equation}	
	\end{proposition}
	
	\begin{proof}
		Let $v,w\in L^\infty(\omega_T)$.
To establish the existence and uniqueness of $\rho\in W_a(0,T)$ and the corresponding estimate, we can apply Theorem \ref{theoremexistence}, with $f=wy\chi_{\omega}\in L^2(Q)\hookrightarrow L^2((0,T);(H^1_a(\Omega))')$.	
	Now, let $\lambda>0$ and set $\dis y_\lambda:=\frac{y(v+\lambda w)-y(v)}{\lambda}.$ Then,
		$y_\lambda$ is solution to the problem
		\begin{equation*}
		\left\{
		\begin{array}{lllll}
		\dis  (y_\lambda)_t-(a(x)(y_\lambda)_x)_x &=&(vy_\lambda+wy(v+\lambda w))\chi_{\omega } &\hbox{in} & Q,\\
		\dis a(x)(y_\lambda)_x(\cdot,x)|_{x=\pm 1}&=&0  &\mbox{in}& (0,T),\\
		y_\lambda(0,\cdot)&=&0 &\hbox{in}& \Omega.
		\end{array}
		\right.
		\end{equation*}
		Define $p_\lambda:=e^{-rt}(y_\lambda-\rho)$, with $r>0$. Then, $p_\lambda$ is a solution to
		\begin{equation}\label{a3}
		\left\{
		\begin{array}{lllll}
		\dis  (p_\lambda)_t-(a(x)(p_\lambda)_x)_x+rp_\lambda &=&[vp_\lambda+e^{-rt}w(y(v+\lambda w)-y(v))]\chi_{\omega } &\hbox{in} & Q,\\
		\dis a(x)(p_\lambda)_x(\cdot,x)|_{x=\pm 1}&=&0  &\mbox{in}& (0,T),\\
		p_\lambda(0,\cdot)&=&0 &\hbox{in}& \Omega.
		\end{array}
		\right.
		\end{equation}	
			Using the estimation \eqref{estimation01}, we obtain that
		\begin{eqnarray*}
			\|p_\lambda\|_{L^2((0,T);H^1_a(\Omega))}\leq \|w\|_{\infty}\|y(v+\lambda w)-y(v)\|_{L^2((0,T);H^1_a(\Omega))}.	
		\end{eqnarray*}
		By taking the limit as $\lambda\to 0$ in this latter inequality and using the Proposition \ref{prop2}, we obtain that $p_\lambda\to 0$ strongly in $L^2((0,T);H^1_a(\Omega))$. Hence $y_\lambda\to \rho$ strongly in $L^2((0,T);H^1_a(\Omega))$ as $\lambda\to 0$. This proves \eqref{diff1}.
	\end{proof}

	\begin{proposition}[ Fr\'echet differentiability of $\mathcal{J}$]\label{diff4}
		Let $u\in L^\infty(\omega_T)$, $y=G(u)$ be the solution of \eqref{ad11} and $q$ be the solution of \eqref{ad21}.	Under the hypothesis of Propositions \ref{diff}, the functional $\mathcal{J}:L^\infty(\omega_T)\to \R$ is continuously Fr\'echet differentiable. Moreover for every $w\in L^\infty(\omega_T)$, we have
		\begin{equation}\label{diff5}
		\mathcal{J}'(u)w=\int_{\omega_T}(\alpha u+y q)w\, \dq.
		\end{equation}
	\end{proposition}
	
	\begin{proof}
		Let $u,w\in L^\infty(\omega_T)$, we have after some straightforward calculations and using Proposition \ref{diff0},
		\begin{equation}\label{euler1}
		\lim_{\lambda\to 0}	\frac{\mathcal{J}(u+\lambda w)-\mathcal{J}(u)}{\lambda}=\int_{\Om}\rho(T)(y(T,u)-y^d)\,\dx+\alpha\int_{\omega_T}uw\ \dq,
		\end{equation}	
		where $\rho\in W_a(0,T)$ is the
		unique solution  of \eqref{diff1}, that we recall here:
	\begin{equation*}
	\left\{
	\begin{array}{lllll}
	\dis  \rho_t-(a(x)\rho_x)_x &=&(u\rho+wy)\chi_{\omega } &\hbox{in} & Q,\\
	\dis a(x)\rho_x(\cdot,x)|_{x=\pm 1}&=&0  &\mbox{in}& (0,T),\\
	\rho(0,\cdot)&=&0 &\hbox{in}& \Omega.
	\end{array}
	\right.
	\end{equation*}	
		To interpret \eqref{euler1}, we use the adjoint state given by \eqref{ad21}, that we recall here
		\begin{equation*}
		\left\{
		\begin{array}{lllll}
		\dis  -q_t-(a(x)q_x)_x&=&uq\chi_{\omega } &in & Q,\\
		\dis a(x)q_x(\cdot,x)|_{x=\pm 1}&=&0  &\mbox{in}& (0,T),\\
		q(T,\cdot)&=&y(T;u)-y^d&in& \Omega.
		\end{array}
		\right.
		\end{equation*}
		
Make the change of variable $t\mapsto T-t$ in this latter problem,  we have that $\varphi(x,t)=q(x,T-t)$  satisfies
		\begin{equation}\label{ad21varphi1}
		\left\{
		\begin{array}{lllll}
		\dis  \varphi_t-(a(x)\varphi_x)_x&=&\tilde{u}\varphi\chi_{\omega } &\hbox{in} & Q,\\
		\dis \dis a(x)\varphi_x(\cdot,x)|_{x=\pm 1}&=&0 &\hbox{in}&(0,T),\\
		\varphi(0,\cdot)&=&y^0-y^d &\hbox{in}& \Omega,
		\end{array}
		\right.
		\end{equation}
	where $\tilde{u}=u(T-t)$.
		Since $y^0-y^d \in L^\infty(\Omega)$ and $\tilde{u}\in L^\infty(\omega_T)$, then applying  Theorem \ref{theoremexistence} with $f=0$ and Corollary \ref{theominmax}, we deduce  that,  there exists a unique adjoint state $q\in W_a(0,T)\cap L^\infty(Q)$ solution to \eqref{ad21} in the sense of Definition \ref{weaksolution}.
		
		So, if we multiply the first equation in \eqref{diff1} by the solution $q$ of \eqref{ad21}, and integrate by parts over $Q$, we get
		\begin{equation}\label{equal1}
		\int_{\Om}y(T)(y(T;u)-y^d)\,\dx=\int_{\omega_T}wy(u)q\ \dq ,
		\end{equation}	
		which combining  with \eqref{euler1} gives
		\begin{eqnarray*}
			\mathcal{J}'(u)w=\int_{\omega_T}(\alpha u+y(u)q)w\ \dq.
		\end{eqnarray*}
		
	\end{proof}
	We now prove the main result of this section.
		\begin{proof}[Proof of Theorem \ref{theoSO}]
			Let $v\in \U$ be arbitrary. Since $\U$ is convex, then $u+\lambda(v-u)\in \U$ for all $\lambda\in (0,1]$. But $u$ is an $L^\infty$-local minimum, so $\mathcal{J}(u+\lambda(v-u))\geq \mathcal{J}(u)$ and hence 
			$$\frac{\mathcal{J}(u+\lambda(v-u))-\mathcal{J}(u)}{\lambda}\geq 0\qquad \forall \;\lambda\in (0,1].$$
By letting $\lambda\to 0$ in this latter inequality, we obtain 	\eqref{ineq}.
		We have already shown \eqref{ad11} in Theorem \ref{existcontrol1}. To complete the proof of the Theorem \ref{theoSO}, we write 
		\begin{eqnarray*}
		\lim_{\lambda\to 0}	\frac{\mathcal{J}(u+\lambda w)-\mathcal{J}(u)}{\lambda} \geq 0\quad \forall v\in \mathcal{U}.
		\end{eqnarray*}	
		Using Proposition \ref{diff4}, with $w=v-u$, we obtain that
		%
		\begin{eqnarray*}
			\int_{\omega_T}(\alpha u+y(u)q)(v-u)\ \dq\geq 0 \quad \forall v\in \mathcal{U},
		\end{eqnarray*}
	where $q$ is the solution of \eqref{ad21}.
			
	\end{proof}

\begin{remark}\label{remark4}
	The relation \eqref{ineq} (equivalently \eqref{neccoptcond}) is called first order necessary conditions.
	We notice that from Theorem \ref{theoSO}, a necessary and sufficient condition for \eqref{neccoptcond} to hold is that for a.e $(t,x)\in Q$,
	\begin{equation}\label{eq0}
	\left\{
	\begin{array}{lllll}
	\dis u=  m &if& \alpha u(t,x)+y(u)q(t,x)>0,\\
	\dis u\in [m,M]&if&\alpha u(t,x)+y(u)q(t,x)=0,\\
	u=M&if&\alpha u(t,x)+y(u)q(t,x)<0.
	\end{array}
	\right.
	\end{equation}
	Note that from \eqref{eq0}, we have that $|\alpha u+y(u)q|>0$ implies $u=m$ or $u=M$.
\end{remark}

	\begin{remark}\label{contadj}
		By a change of variable $t\mapsto T-t$, one can show that the solution $q$ of the adjoint state \eqref{ad21} which belongs to $W_a(0,T)\cap L^\infty(Q)$ satisfies the estimates:
		\begin{equation}\label{estimation2}
\|q\|_{\C([0,T];L^2(\Omega))}+ \|q\|_{L^2((0,T);H^1_a(\Omega))} \leq C(\|u\|_{\infty},T)\left[\|y^0\|_{L^\infty(\Omega) }+\|y^d\|_{L^\infty(\Omega) }\right]
		\end{equation}
		and 
		\begin{equation}\label{qinfini}
		\|q\|_{L^\infty(Q)} \leq e^{(\|u\|_{\infty}+1)T}\left[\|y^0\|_{L^\infty(\Omega) }+\|y^d\|_{L^\infty(\Omega) }\right].
		\end{equation}
	\end{remark}
We end this section by the following important results.

\begin{lemma}\label{cont1}
For any $u\in L^\infty(\omega_T)$, the linear mapping $v\mapsto \mathcal{J}'(u)v$ can be extended to a continuous linear mapping $\mathcal{J}'(u):L^2(\omega_T)\to \mathbb{R}$, as defined by \eqref{diff5}.
\end{lemma}	
\begin{proof} Let $u\in L^\infty(\omega_T)$ and $v\in L^2(\omega_T)$. From \eqref{diff5}, we derive the expression : 
	\begin{equation*}
	\mathcal{J}'(u)v=\int_{\omega_T}(\alpha u+y q)v\, \dq,
	\end{equation*}
	where $y$ and $q$ are solutions to \eqref{ad11} and \eqref{ad21}, respectively.
	Then using \eqref{minmax} and \eqref{estimation2}, there exists a constant $C=C\left(\alpha, \|u\|_{\infty},T, \|y^0\|_{L^\infty(\Omega)}, \|y^d\|_{L^\infty(\Omega)}\right)$ independent of $v$ such that
	\begin{equation*}
	\begin{array}{lll}
	\dis |\mathcal{J}'(u)v|\leq C\|v\|_{L^2(\Omega)}.
	\end{array}
	\end{equation*}
Hence, the 	mapping $v\mapsto \mathcal{J}'(u)v$ is a linear continuous mapping on $L^2(\omega_T)$.
\end{proof}

Arguing as the same as in the proof of Proposition \ref{prop2}, we obtain the following result.
\begin{proposition}\label{prop3} Let  $u\in L^\infty(\omega_T)$. Then the mapping  $u\mapsto q(u)$ solution of the adjoint problem \eqref{ad21} is a locally Lipschitz continuous function from $L^2(\omega_T)$ into $W_a(0,T)$. More precisely for all $u_1,u_2\in L^\infty(\omega_T)$, there exists a constant $C=C(\|u_1\|_{\infty},\|u_2\|_{\infty}, \|y^d\|_{L^\infty(\Omega)},\|y^0\|_{L^\infty(\Omega)},T)>0$  such that the following estimate 
	\begin{equation}\label{estim21}
\|q(u_1)-q(u_2)\|_{\C([0,T];L^2(\Omega))}+	\|q(u_1)-q(u_2)\|_{L^2((0,T);H^1_a(\Omega))}\leq C \|u_1-u_2\|_{L^2(\omega_T)}
	\end{equation}
	holds.
\end{proposition}

	\section{Second-order necessary and sufficient optimality conditions }
	\label{sufficientoptcond}
	Since the cost functional \eqref{defJ} associated to the optimization problem \eqref{optimal} is non-convex, the first order optimality conditions developed in Theorem \ref{theoSO} are necessary but not sufficient for optimality. In this section, we develop the second order optimality conditions which ensure the sufficiency. They also ensure the stability of optimal solutions with respect to perturbations of the problems such as finite element discretization \cite{casas2015a}.
	
	By employing the same arguments as presented in the proof of Proposition \ref{diff}, we arrive at the following result.
	
	\begin{proposition}[Twice Fr\'echet differentiability of $G$]\label{diff2}
		Under the hypothesis of Proposition \ref{diff}, the control-to-state mapping $G:v\mapsto y$, is twice continuously Fr\'echet differentiable from $L^\infty(\omega_T)$ into $W_a(0,T)$. Moreover, the second derivative of $G$ at $u\in L^\infty(\omega_T)$ is given by the expression
		
		$$G''(u)[w,h]=z,$$
		where $w,h\in L^\infty(\omega_T)$ and with $z\in W_a(0,T)$ being the uniquely determined weak solution of
		\begin{equation}\label{diff3}
		\left\{
		\begin{array}{lllll}
		\dis  z_t-(a(x)z_x)_x  &=&(uz+hG'(u)w+wG'(u)h)\chi_{\omega } &\hbox{in} & Q,\\
		\dis a(x)z_x(\cdot,x)|_{x=\pm 1}&=&0  &\mbox{in}& (0,T),\\
		z(0,\cdot)&=&0 &\hbox{in}& \Omega.
		\end{array}
		\right.
		\end{equation}		
	\end{proposition}

\begin{proposition}[Twice Fr\'echet differentiability of $\mathcal{J}$]\label{diff44}
	Let $u\in L^\infty(\omega_T)$ and $y=G(u)$ be the solution of \eqref{ad11} and $q$ be the solution of \eqref{ad21}.	Under the hypothesis of Propositions \ref{diff2} and \ref{diff}, the functional $\mathcal{J}:L^\infty(\omega_T)\to \R$ is twice continuously Fr\'echet differentiable. Moreover for every $u,w, h\in L^\infty(\omega_T)$, we have
	\begin{equation}\label{diff6}
	\mathcal{J}''(u)[w,h]=\int_{\omega_T}[hG'(u)w+wG'(u)h]q\, \dq+\int_{\Om}(G'(u)w)(T)(G'(u)h)(T)\, \dx+\alpha\int_{\omega_T}hw\, \dq.
	\end{equation}
\end{proposition}

\begin{proof}
	The result follows from Lemma \ref{lemmeG}, Proposition \ref{diff2}, straightforward computations and Lebesgue dominated convergence theorem.
\end{proof}
We have the following important result.
\begin{lemma}\label{diff}
	Let $G:L^\infty(\omega_T)\to W_a(0,T)\cap L^\infty(Q), u\mapsto y$ be the control-to-state operator,  where $y$ is solution to \eqref{model}. Then for any $u\in L^\infty(\omega_T)$, the linear mapping $v\mapsto G'(u)v$ can be extended to a linear continuous mapping from $L^2(\omega_T)\to W_a(0,T)$.
	In addition, the following estimate holds
	\begin{equation}\label{estimation11}
	\|G'(u)w\|_{\C([0,T];L^2(\Omega))}+	\|G'(u)w\|_{L^2((0,T);H^1_a(\Omega))}\leq 2e^{2(\|u\|_{\infty}+1)T}\|w\|_{L^2(\omega_T)},
	\end{equation}	
	for all $v\in L^2(\omega_T).$
\end{lemma}	
\begin{proof} For the extension, it is sufficient to prove that for any $v\in L^2(\omega_T)$, the problem \eqref{diff1} has a unique solution $\rho\in W_a(0,T)$. This follows directly from Theorem \ref{theoremexistence}, with $f=wy\chi_{\omega}\in L^2(Q)\hookrightarrow L^2((0,T);(H^1_a(\Omega))')$. Note that $f\in L^2(Q)$ because $y\in L^\infty(Q)$. For the estimate, we take $y^0=0$ and $f=wy\chi_{\omega }$ in \eqref{estimation0}, to obtain 
	$$\|G'(u)w\|_{\C([0,T];L^2(\Omega))}+	\|G'(u)w\|_{L^2((0,T);H^1_a(\Omega))}\leq 2e^{(\|u\|_{\infty}+1)T}\|w\|_{L^2(\omega_T)}\|y\|_{L^\infty(Q)}, $$
	which using the estimate \eqref{minmax} gives
	$$\|G'(u)w\|_{\C([0,T];L^2(\Omega))}+	\|G'(u)w\|_{L^2((0,T);H^1_a(\Omega))}\leq 2e^{2(\|u\|_{\infty}+1)T}\|w\|_{L^2(\omega_T)}.$$ 
	
\end{proof}

\begin{lemma}\label{cont2}
	Let $u\in L^\infty(\omega_T)$, then, the bilinear mapping $(w,h)\mapsto \mathcal{J}''(u)[w,h]$ can be extended to a bilinear continuous mapping on $\mathcal{J}''(u):L^2(\omega_T)\times L^2(\omega_T)\to \R$ given by \eqref{diff6}.
\end{lemma}	
\begin{proof} $u\in L^\infty(\omega_T)$ and $w,h\in L^2(\omega_T)$. From \eqref{diff6}, we have 
	\begin{equation*}
\mathcal{J}''(u)[w,h]=\int_{\omega_T}[hG'(u)w+wG'(u)h]q\, \dq+\int_{\Om}(G'(u)w)(T)(G'(u)h)(T)\, \dx+\alpha\int_{\omega_T}hw\, \dq,
	\end{equation*}
	where $G'(u)w$ is solution to \eqref{diff1}.
Using Cauchy Schwarz's inequality, \eqref{estimation2} and \eqref{estimation11}, we obtain

\begin{equation*}
\begin{array}{lll}
	\dis |\mathcal{J}''(u)[w,h]|&\leq& \|q\|_{L^\infty(Q)}\left[\|h\|_{L^2(\omega_T)}\|G'(u)w\|_{L^2(Q)}+\|w\|_{L^2(\omega_T)}\|G'(u)h\|_{L^2(Q)}\right]\\
	&+& \|(G'(u)w)(T)\|_{L^2(\Omega)}\|(G'(u)h)(T)\|_{L^2(\Omega)}+\alpha \|w\|_{L^2(\omega_T)}\|h\|_{L^2(\omega_T)}\\
	&\leq& C\|w\|_{L^2(\omega_T)}\|h\|_{L^2(\omega_T)}.
\end{array}
\end{equation*}
Therefore, there exists a constant $C=C(\alpha,\|u\|_{\infty},\|y^0\|_{L^\infty(\Omega) },\|y^d\|_{L^\infty(\Omega)},T)>0$ independent of $v$ and $h$ such that
	\begin{equation*}
	\begin{array}{lll}
	\dis |\mathcal{J}''(u)[w,h]|\leq C\|w\|_{L^2(\omega_T)}\|h\|_{L^2(\omega_T)}.
	\end{array}
	\end{equation*}
	Hence, the 	mapping $(w,h)\mapsto\mathcal{J}''(u)[w,h]$ is a bilinear continuous mapping on $L^2(\omega_T)\times L^2(\omega_T)$.
\end{proof}

	Before stating second order necessary and sufficient conditions, we first introduce some preliminary concepts retrieved from \cite{fredi2010}. 
	
	\begin{definition}
		$ $
		\begin{enumerate}[-]
			\item The set of strongly active constraints is the set $A_\tau(u)$ defined by
			\begin{equation*}
			A_\tau(u):=\{(t,x)\in \omega_T : |u(t,x)+y(u)q(t,x)|>\tau\}.
			\end{equation*}
			\item We introduce the following
			for a.e $(t,x)\in \omega_T$,
			\begin{equation}\label{eq1}
			v(t,x)\left\{
			\begin{array}{lllll}
			\dis  \geq 0 &if&  u(t,x)=m,\\
			\dis \leq 0&if& u(t,x)=M,\\
			0&if&(t,x)\in A_\tau(u).
			\end{array}
			\right.
			\end{equation}
			The $\tau$-critical (see e.g \cite{casas2015a,fredi2010}) associated to a control $u$ is defined by 
			\begin{equation}\label{ccone}
			C_\tau (u)= \{v\in L^2(\omega_T) : v \;\text{fulfills}\; \eqref{eq1}\}.
			\end{equation}
		\end{enumerate}
	\end{definition}
	In the rest of the paper, we will adopt the following notation $\mathcal{J}''(u)v^2:=\mathcal{J}''(u)[v,v]$.
	
	\begin{proposition}[Second order necessary optimality conditions]
		Let $u\in \mathcal{U}$ be a $L^\infty$-local solution of problem \eqref{optimal}. Then $\mathcal{J}''(u)v^2\geq 0$ for all $v\in C_0 (u)$.
	\end{proposition}
\begin{proof}
The proof follows similar arguments as \cite{kenne2022bil,fredi2010}. Let $v\in C_0 (u)$ and $\theta\in (0,1)$. For $n\in \N$, we define the set 
$$I_n=\left\{(x,t)\in \omega_T: m+\frac{1}{n}\leq u(x,t)\leq M-\frac{1}{n}\right\}.$$
Let $v_n:=\chi_n v$, where
\begin{equation*}
\chi_n(x,t)=\left\{
\begin{array}{lllll}
\dis  1 &if&  (x,t)\in I_n \;\text{or}\; u(x,t)\in \{m,M\} \;\text{and}\; M-m\geq \dis \frac{1}{n} ,\\
0&if& u(x,t)\in \dis \left(m,m+\frac{1}{n}\right)\cup \left(M-\frac{1}{n},M\right).
\end{array}
\right.
\end{equation*}
Then, $\chi_n(x,t)=0$ also if $u(x,t)\in \{m,M\}$ and $\dis M-m< \dis \frac{1}{n}$, and we have $\bar{u}=u+\theta v_n\in \mathcal{U}$ for $\theta\in (0,1)$. Moreover using that $u$ is a locally optimal control, we deduce that 

\begin{eqnarray*}
	0\leq \frac{\mathcal{J}(\bar{u})-\mathcal{J}(u)}{\theta}=\mathcal{J}'(u)v_n+\frac{1}{2}\theta \mathcal{J}''(u)v_n^2+\theta^{-1}r^2(u,\theta v_n),
\end{eqnarray*}
where $r^2(u,\theta v_n)$ represents the second-order remainder. Since $v\in C_0(u)$, then $v_n\in C_0(u)$ and then $\dis \mathcal{J}'(u)v_n=\int_{\omega_T}(\alpha u-\rho q)v_n\, \dq=0$. Therefore,
by dividing the both inequalities of the previous identity by $\theta$, we obtain

\begin{eqnarray*}
	0\leq \frac{1}{2}\mathcal{J}''(u)v_n^2+\theta^{-2}r^2(u,\theta v_n).
\end{eqnarray*}

Taking the limit as $\theta\to 0$ in the last inequality yields 
\begin{equation}\label{1}
\mathcal{J}''(u)v_n^2\geq 0.
\end{equation} 
It remains to prove that as $n\to +\infty$, $v_n\to v$ in $L^2(\omega_T)$. First we note that for a.e $(t,x)\in \omega_T$, $v_n(t,x)\to v(t,x)$ pointwise almost everywhere as $n\to +\infty$. In addition $|v_n(t,x)|^2\leq |v(t,x)|^2$ pointwise everywhere for all $n\in \N$. Then, using the Lebesgue's dominated convergence theorem, we deduce that as $n\to +\infty$, $v_n\to v$ in $L^2(\omega_T)$. Hence, by taking the limit as $n\to +\infty$ in \eqref{1} and using the continuity of $v\to \mathcal{J}''(u)v^2$ in $L^2(\omega_T)$, we obtain $\mathcal{J}''(u)v^2\geq 0$.
\end{proof}
	
	\begin{theorem}\label{hypothesis}
		Let $u\in \mathcal{U}$ be a control satisfying the first order optimality conditions \eqref{ineq}. Then the following hold:
		\begin{enumerate}[(1)]
			\item The functional $\mathcal{J}:L^\infty(\omega_T)\to \R$ is of class $\mathcal{C}^2$. Furthermore, for every $u\in \mathcal{U}$, there exist continuous extensions
			\begin{equation}\label{h1}
			\mathcal{J}'(u)\in \mathcal{L}(L^2(\omega_T),\R)\;\;\;\;\text{and}\;\;\; \mathcal{J}''(u)\in \mathcal{B}(L^2(\omega_T),\R).
			\end{equation}
			\item For any sequence $\left\{(u_k,v_k)\right\}_{k=1}^{\infty}\subset \mathcal{U}\times L^2(\omega_T)$ with $\|u_k-u\|_{L^2(\omega_T)}\to 0$ and $v_k \rightharpoonup v$ weakly in $L^2(\omega_T)$,
			\begin{equation}\label{h2}
			\dis	\mathcal{J}'(u)v=\lim_{k\to+ \infty} \mathcal{J}'(u_k)v_k,
			\end{equation}
			
			\begin{equation}\label{h3}
			\dis	\mathcal{J}''(u)v^2\leq \liminf_{k\to +\infty} \mathcal{J}''(u_k)v_k^2,
			\end{equation}
			
			\begin{equation}\label{h4}
			\dis	\text{If}\; v=0, \;\; \text{then} \;\;\;\Lambda \liminf_{k\to+ \infty}\|v_k\|^2_{L^2(\omega_T)}\leq\liminf_{k\to +\infty} \mathcal{J}''(u_k)v_k^2,
			\end{equation}
			for some $\Lambda>0$.
		\end{enumerate}
	\end{theorem}

\begin{proof}  	
	Using Proposition \ref{diff4}, Lemma \ref{cont1}, Proposition \ref{diff44} and Lemma \ref{cont2}, we obtain the first point $\textit{(1)}$ of Theorem \ref{hypothesis}.\\
		Now, we prove the point $\textit{(2)}$ in three steps.\\
		Let $\left\{(u_k,v_k)\right\}_{k=1}^{\infty}$ be a sequence of $\mathcal{U}\times L^2(\omega_T)$ such that $\|u_k-u\|_{L^2(\omega_T)}\to 0$ and $v_k \rightharpoonup v$ weakly in $L^2(\omega_T)$.\\
		\noindent \textbf{Step 1.} We prove \eqref{h2}. \\
		 Using the Lipschitz continuity property of $G$ given in Proposition \ref{prop2}, we obtain that $\dis G(u_k)\to G(u)$ in $L^2((0,T);H^1_a(\Omega))$ as $k\to \infty$. We also obtain due to Proposition \ref{prop3} that  the solution of problem \eqref{ad21} satisfies $\dis q(u_k)\to q(u)$ in $L^2((0,T);H^1_a(\Omega))$ as $k\to +\infty$. Since, $G(u_k)$ and $G(u)$ are respectively solutions to \eqref{model} with $v=u_k$ and $v=u$. Therefore, using \eqref{minmaxy}, we deduce that in particular the sequence $G(u_k)$ is bounded in $L^2(Q)$ and $G(u)\in L^\infty(Q)$. Hence $G(u_k)q(u_k),G(u)q(u)\in L^2(Q)$ and 
		\begin{equation*}\label{conv}
		G(u_k)q(u_k)\to G(u)q(u) \;\;\text{in}\;\;L^2(Q)\;\;\text{as}\;\; k\to +\infty.
		\end{equation*}
	From the expression of $\mathcal{ J}'$ given in \eqref{diff5}, we deduce using the latter convergence that:
		\begin{eqnarray*}
			\dis \lim_{k\to +\infty} \mathcal{J}'(u_k)v_k=\lim_{k\to+ \infty}\int_{\omega_T}(\alpha u_k+G(u_k)q(u_k))v_k\, \dq =\int_{\omega_T}(\alpha u+G(u)q(u))v\, \dq= \mathcal{J}'(u)v.
		\end{eqnarray*}
		Thus \eqref{h2} is proved.\\
		\noindent \textbf{Step 2.} Let us show \eqref{h3}. We have
		\begin{equation}\label{d}
		\mathcal{J}''(u_k)v_k^2=2\int_{\omega_T}v_k(G'(u_k)v_k)q(u_k)\dq+\int_{\Om}|(G'(u_k)v_k)(T)|^2\, \dx+\alpha\int_{\omega_T}|v_k|^2\, \dq.
		\end{equation}
		Note that $G'(u_k)v_k$ is the unique weak solution of \eqref{diff1} with $u=u_k$ and $w=v_k$. Then, we claim that 
		\begin{equation*}
		(G'(u_k)v_k) q(u_k) \;\;\text{converges strongly to    } (G'(u)v)q(u) \;\; \text{in  } L^2(Q)\;\;\text{ as   } k\to+ \infty.
		\end{equation*}
		The latter convergence follows from the boundedness of the sequences $(G(u_k), q(u_k))$ and $v_k$ in $L^{\infty}(Q)\times L^{\infty}(Q)$ and in $L^2(\omega_T)$ respectively,  the estimation  \eqref{estimation11}, Lemma \ref{embed} and the convergence $\dis q(u_k)\to q(u)$ in $L^2((0,T);H^1_a(\Omega))$ as $k\to +\infty$. Therefore, taking the limit as $k\to +\infty$ in \eqref{d} and using the lower-semi continuity of the $L^2$-norm, we deduce that
		\begin{eqnarray*}
			\dis \lim_{k\to +\infty}  \mathcal{J}''(u_k)v_k^2
			\!\!\!&\geq&\!\!\! 2\lim_{k\to+ \infty}\int_{\omega_T}v_k(G'(u_k)v_k)q(u_k)\dq+\liminf_{k\to+ \infty}\left[\int_{\Om}|(G'(u_k)v_k)(T)|^2\, \dx+\alpha \int_{\omega_T}|v_k|^2\, \dq\right]\\
			&\geq& 2\int_{\omega_T}v(G'(u)v)q(u)\dq+\int_{\Om}|(G'(u)v)(T)|^2\, \dx+\alpha \int_{\omega_T}|v|^2\, \dq\\
			&=& \mathcal{J}''(u)[v,v].
		\end{eqnarray*}
		Hence \eqref{h3} holds.
		
		\noindent \textbf{Step 3.} Finally, we show \eqref{h4}.
		If $v=0$, then in \eqref{d}, the first and the second terms tend to $0$, except the last one. Hence,
		$$\Lambda  \liminf_{k\to+ \infty}\|v_k\|^2_{L^2(\omega_T)}\leq   \lim_{k\to +\infty}  \mathcal{J}''(u_k)v_k^2,$$
		with $\Lambda=\alpha$. 
	\end{proof}

Now, we state one of the main results of this paper.

%
%

	\begin{theorem}[Second order sufficient optimality conditions]\label{Quadratic growth}
	 Let $u\in \mathcal{U}$ be a control satisfying the first order optimality conditions \eqref{neccoptcond} and
		\begin{equation}\label{cdt1}
		\mathcal{J}''(u)v^2> 0 \quad \forall v\in C_0(u)\backslash\{0\}.
		\end{equation}
		
		Then, there are two constants $\varepsilon>0$ and $\gamma>0$ such that the quadratic growth condition
		\begin{equation}\label{cdt2}
		\mathcal{J}(v)\geq \mathcal{J}(u)+\frac{\gamma}{2}\|v-u\|^2_{L^2(\omega_T)}\quad \forall v\in \mathcal{U}\cap B^2_{\varepsilon}(u)
		\end{equation}
			holds. Hence $u$ is locally optimal in the sense of $L^2(\omega_T)$.
	\end{theorem}
	\begin{proof}
	From Theorem  \ref{hypothesis}, the assumptions of  \cite[Theorem 2.3 ]{casas2012} are fulfilled.  Hence,  \eqref{cdt2} holds. 
	\end{proof}
To provide a more detailed understanding of the challenges associated with second-order optimality conditions, we can refer to the following example, which has been extracted from \cite{casas2012}.
	\begin{example}
Let us consider the following optimization problem
\begin{equation}\label{minpb}
\begin{array}{llll}
\dis \min_{v\in L^2(0,1)}J(v)=:\int_{0}^{1}\left[\left(v(x)+\frac{\pi}{2}\right)^2+\sin(u(x))\right]\dx.
\end{array}
\end{equation}
Then, $\dis u(x)=-\frac{\pi}{2}$ is a global solution of the optimization problem \eqref{minpb}. In addition, $J$ is of class $\mathcal{C}^2$ in $L^\infty(0,1)$, $J'(u)v=0$ and $J''(u)v^2=3\|v\|^2_{L^2(0,1)}$, $\forall v\in L^2(0,1)$. However, it is easy to check that the assumptions of Theorem \ref{hypothesis} are satisfied; thus Theorem \ref{Quadratic growth} implies that $u$ is a strict local minimum in the sense of $L^2(0, 1)$.
	\end{example}

\section*{Acknowledgements}
We would like to thank the reviewers for their valuable comments and suggestions which helped us to improve significantly the paper.
\bibliographystyle{abbrv}
\bibliography{mybibfile}

\begin{thebibliography}{10}

\bibitem{arab2022}
Z.~Arab and C.~Tun{\c{c}}.
\newblock Well-posedness and regularity of some stochastic time-fractional
  integral equations in hilbert space.
\newblock {\em Journal of Taibah University for Science}, 16(1):788--798, 2022.

\bibitem{aronna2012}
M.~S. Aronna.
\newblock Singular solutions in optimal control: second order conditions and a
  shooting algorithm.
\newblock {\em arXiv preprint arXiv:1210.7425}, 2012.

\bibitem{aronna2016}
M.~S. Aronna, J.~F. Bonnans, and B.~S. Goh.
\newblock Second order analysis of control-affine problems with scalar state
  constraint.
\newblock {\em Mathematical Programming}, 160(1):115--147, 2016.

\bibitem{aronna2018}
M.~S. Aronna, J.~F. Bonnans, and A.~Kr\"{o}ner.
\newblock Optimal control of infinite dimensional bilinear systems: application
  to the heat and wave equations.
\newblock {\em Mathematical Programming}, 168(1):717--757, 2018.

\bibitem{aronna2019}
M.~S. Aronna, J.~F. Bonnans, and A.~Kroner.
\newblock Optimal control of pdes in a complex space setting: Application to
  the schr\"{o}dinger equation.
\newblock {\em SIAM Journal on Control and Optimization}, 57(2):1390--1412,
  2019.

\bibitem{aronna2021ss}
M.~S. Aronna, J.~F. Bonnans, and A.~Kr{\"o}ner.
\newblock State-constrained control-affine parabolic problems i: first and
  second order necessary optimality conditions.
\newblock {\em Set-Valued and Variational Analysis}, 29(2):383--408, 2021.

\bibitem{aronna2021s}
M.~S. Aronna, J.~F. Bonnans, and A.~Kr\"{o}ner.
\newblock State constrained control-affine parabolic problems ii: Second order
  sufficient optimality conditions.
\newblock {\em SIAM Journal on Control and Optimization}, 59(2):1628--1655,
  2021.

\bibitem{aronna2021}
M.~S. Aronna and F.~Tr{\"o}ltzsch.
\newblock First and second order optimality conditions for the control of
  fokker-planck equations.
\newblock {\em ESAIM: Control, Optimisation and Calculus of Variations}, 27:15,
  2021.

\bibitem{budyko1969}
M.~I. Budyko.
\newblock The effect of solar radiation variations on the climate of the earth.
\newblock {\em tellus}, 21(5):611--619, 1969.

\bibitem{cannarsa2011}
P.~Cannarsa and G.~Floridia.
\newblock Approximate controllability for linear degenerate parabolic problems
  with bilinear control.
\newblock {\em arXiv preprint arXiv:1106.4232}, 2011.

\bibitem{cannarsa2011a}
P.~Cannarsa and G.~Floridia.
\newblock Approximate multiplicative controllability for degenerate parabolic
  problems with robin boundary conditions.
\newblock {\em Communications in Applied and Industrial Mathematics},
  2(2):1--16, 2011.

\bibitem{casas2012}
E.~Casas and F.~Tr{\"o}ltzsch.
\newblock Second order analysis for optimal control problems: improving results
  expected from abstract theory.
\newblock {\em SIAM Journal on Optimization}, 22(1):261--279, 2012.

\bibitem{casas2015a}
E.~Casas and F.~Tr{\"o}ltzsch.
\newblock Second order optimality conditions and their role in pde control.
\newblock {\em Jahresbericht der Deutschen Mathematiker-Vereinigung},
  117(1):3--44, 2015.

\bibitem{diaz2006}
J.~Diaz, G.~Hetzer, and L.~Tello.
\newblock An energy balance climate model with hysteresis.
\newblock {\em Nonlinear Analysis: Theory, Methods \& Applications},
  64(9):2053--2074, 2006.

\bibitem{diaz1999}
J.~I.~D. D{\'\i}az and L.~T. del Castillo.
\newblock A nonlinear parabolic problem on a riemannian manifold without
  boundary arising in climatology.
\newblock {\em Collectanea Mathematica}, pages 19--51, 1999.

\bibitem{dorville2021}
R.~Dorville.
\newblock Bilinear boundary optimal control with final observation for the heat
  equation.
\newblock {\em Applicable Analysis}, 2021.

\bibitem{epstein2013}
C.~L. Epstein and R.~Mazzeo.
\newblock {\em Degenerate diffusion operators arising in population biology}.
\newblock Number 185. Princeton University Press, 2013.

\bibitem{floridia2014}
G.~Floridia.
\newblock Approximate controllability for nonlinear degenerate parabolic
  problems with bilinear control.
\newblock {\em Journal of Differential Equations}, 257(9):3382--3422, 2014.

\bibitem{floridia2020}
G.~Floridia.
\newblock Nonnegative controllability for a class of nonlinear degenerate
  parabolic equations with application to climate science.
\newblock {\em arXiv preprint arXiv:2003.04966}, 2020.

\bibitem{ghil1976}
M.~Ghil.
\newblock Climate stability for a sellers-type model.
\newblock {\em Journal of Atmospheric Sciences}, 33(1):3--20, 1976.

\bibitem{cyrillek2022}
C.~Kenne.
\newblock {\em Sur les mod{\`e}les de dynamique de populations et
  l'{\'e}mergence de la maladie dans les eaux douces}.
\newblock PhD thesis, Universit\'e des Antilles, 2022.

\bibitem{kenne2020}
C.~Kenne, G.~Leugering, and G.~Mophou.
\newblock Optimal control of a population dynamics model with missing birth
  rate.
\newblock {\em SIAM Journal on Control and Optimization}, 58:1289--1313, 2020.

\bibitem{kenne2022bil}
C.~Kenne, G.~Mophou, and M.~Warma.
\newblock Bilinear optimal control for a fractional diffusive equation.
\newblock {\em arXiv preprint arXiv:2210.17494}, 2022.

\bibitem{cb2021}
C.~Kenne and B.~Nkemzi.
\newblock Optimal control of averaged state of a population dynamics model.
\newblock In {\em Studies in Evolution Equations and Related Topics}, pages
  113--127. Springer, 2021.

\bibitem{cprg2021}
C.~Kenne, P.~Zongo, R.~Dorville, and G.~Mophou.
\newblock Optimal control of a coupled degenerate population dynamics model
  with unknown birth rates.
\newblock {\em Nonlinear Studies (NS)}, 28(4):1225--1252, 2021.

\bibitem{lions2013}
J.~L. Lions.
\newblock {\em Equations diff{\'e}rentielles op{\'e}rationnelles: et
  probl{\`e}mes aux limites}, volume 111.
\newblock Springer-Verlag, 2013.

\bibitem{landry2020}
G.~Mophou, M.~K{\'e}r{\'e}, and L.~L.~D. Njoukou{\'e}.
\newblock Robust hierarchic control for a population dynamics model with
  missing birth rate.
\newblock {\em Mathematics of Control, Signals, and Systems}, 32(2):209--239,
  2020.

\bibitem{north1975}
G.~R. North.
\newblock Analytical solution to a simple climate model with diffusive heat
  transport.
\newblock {\em Journal of Atmospheric Sciences}, 32(7):1301--1307, 1975.

\bibitem{north2017}
G.~R. North and K.-Y. Kim.
\newblock {\em Energy Balance Climate Models}.
\newblock John Wiley \& Sons, 2017.

\bibitem{salim2023}
A.~Salim, F.~Mesri, M.~Benchohra, and C.~Tun{\c{c}}.
\newblock Controllability of second order semilinear random differential
  equations in fr{\'e}chet spaces.
\newblock {\em Mediterranean Journal of Mathematics}, 20(2):84, 2023.

\bibitem{sellers1969}
W.~D. Sellers.
\newblock A global climatic model based on the energy balance of the
  earth-atmosphere system.
\newblock {\em Journal of Applied Meteorology (1962-1982)}, pages 392--400,
  1969.

\bibitem{shimakura1992}
N.~Shimakura.
\newblock {\em Partial differential operators of elliptic type}.
\newblock Amer Mathematical Society, 1992.

\bibitem{tory2003}
E.~M. Tory, K.~H. Karlsen, R.~B{\"u}rger, and S.~Berres.
\newblock Strongly degenerate parabolic-hyperbolic systems modeling
  polydisperse sedimentation with compression.
\newblock {\em SIAM Journal on Applied Mathematics}, 64(1):41--80, 2003.

\bibitem{fredi2010}
F.~Tr{\"o}ltzsch.
\newblock {\em Optimal control of partial differential equations: theory,
  methods, and applications}, volume 112.
\newblock American Mathematical Soc., 2010.

\bibitem{tuncc2023}
C.~Tun{\c{c}}, O.~Tun{\c{c}}, C.-F. Wen, and J.-C. Yao.
\newblock On the qualitative analyses solutions of new mathematical models of
  integro-differential equations with infinite delay.
\newblock {\em Mathematical Methods in the Applied Sciences}, 2023.

\end{thebibliography}
	
\end{document}